\documentclass{article}

\usepackage[english]{babel} 
\usepackage{afterpage}
\usepackage{algorithm}
\usepackage{algorithmicx}
 \usepackage{algpseudocode}
 \usepackage{authblk}
\usepackage{amsmath}
\usepackage{amsthm}
\usepackage{mathtools}
\usepackage[title]{appendix}
\usepackage{amsfonts}
\usepackage{amssymb}
\usepackage{booktabs} 
\usepackage{chngcntr}
\usepackage{comment} 
\usepackage{csquotes} 
\usepackage{csvsimple} 
\usepackage{diagbox}
\usepackage{enumitem}
\usepackage{float} 
\usepackage[T1]{fontenc} 
\usepackage{geometry}
\usepackage{graphicx}
\usepackage[colorlinks=true, allcolors=blue]{hyperref}
\usepackage[nopatch=footnote]{microtype}
\usepackage{placeins}
\usepackage{needspace}
\usepackage{listings}
\usepackage{lipsum}
\usepackage{longtable}
\usepackage{mathrsfs}
\usepackage{multirow}
\usepackage{orcidlink}
\usepackage{pdflscape}
\usepackage[]{ragged2e}

\usepackage{setspace} 
\usepackage{threeparttable} 
\usepackage{tabularray}
\usepackage{cite}
\usepackage{subcaption}

\hypersetup{hypertexnames=false} 

\theoremstyle{plain}         
\newtheorem{remark}{Remark}[section]   
\theoremstyle{plain}         
\newtheorem{example}{Example}[section]   
\theoremstyle{plain}          
\newtheorem{theorem}{Theorem}[section] 
\newtheorem{lemma}{Lemma}[section] 
\newtheorem{assumption}{Assumption}[section] 
\newtheorem{definition}{Definition}[section] 
\newtheorem{proposition}{Proposition}[section] 

\graphicspath{{figures/}} 

\title{Gradient-Enhanced Proximal Algorithms for Mean Field Planning on Surfaces}
\author[1]{Chengrun Jiang}
\affil[1]{Department of Mathematical Sciences, Tsinghua University, Beijing, 100084,  China (\href{mailto:jcr22@mails.tsinghua.edu.cn}{jcr22@mails.tsinghua.edu.cn})}
\date{July 26, 2025}

\begin{document}
\maketitle

\begin{abstract}
Mean field planning on a surface prescribes initial and terminal densities and minimizes a transport energy subject to the continuity equation. Proximal algorithms for this problem repeatedly solve a time--space Poisson equation, whose temporal derivative and surface gradient determine the density and momentum corrections. We study a gradient-enhanced approximation of this constraint projection using finite differences in time, surface finite elements in space, temporal polynomial preserving recovery, and spatial parametric polynomial preserving recovery. The same construction is incorporated into ISTA, FISTA, and Douglas--Rachford splitting. We distinguish the recovered update from an exact discrete projection and derive residual identities and conditional finite-iteration perturbation bounds that retain data, boundary, and linear-solver errors. Existing derivative-recovery estimates identify a higher-order contribution under suitable regularity and mesh assumptions; they do not by themselves establish convergence of the outer optimization iteration. Available numerical illustrations on the sphere and a more complicated algebraic surface are discussed together with the limits of the recorded refinement data.
\end{abstract}

\section{Introduction}

\begingroup
\setlength{\emergencystretch}{2em}

Mean field games (MFGs) describe strategic interactions among large populations of agents whose individual decisions depend on the population distribution. Foundational developments by Lasry and Lions~\cite{lasry2007mean} and Huang, Malham\'e, and Caines~\cite{huang2006large} connect individual optimal control with macroscopic density evolution. Mean field planning (MFP) concerns the related problem of connecting prescribed initial and terminal population distributions; its analysis includes the diffusive planning problem studied by Porretta~\cite{porretta2014planning}. Here we focus on first-order models without diffusion and with a potential structure: the coupling derives from a functional of the density, and the planning problem admits a variational formulation. Convex duality also underlies the weak-solution theory for first-order MFGs with local couplings developed by Cardaliaguet and Graber~\cite{cardaliaguet2015firstorder}.

The connection with optimal transport (OT) is transparent in density--momentum variables. The Benamou--Brenier formulation~\cite{benamou2000computational} minimizes kinetic action subject to the continuity equation and prescribed endpoint densities. Adding an interaction functional leads to variational MFP; Orrieri, Porretta, and Savar\'e~\cite{orrieri2019variational} developed a rigorous first-order planning theory using convex duality and dynamic transport. We consider the manifold formulation~\eqref{equ:MFP_main_form}, where $\rho$ is the density, $\mathbf{m}=\rho\mathbf{v}$ is the tangent momentum, and $\mathcal{F}$ describes the interaction. Setting $\mathcal{F}=0$ and $L(\rho,\mathbf{m})=|\mathbf{m}|^2/(2\rho)$ for $\rho>0$, with its lower semicontinuous extension at vacuum, recovers dynamic OT. Potential MFGs instead admit a terminal-cost formulation such as~\eqref{equ:MFG_main_form}.

Several numerical approaches exploit these structures. Achdou and Capuzzo-Dolcetta~\cite{achdou2010numerical} introduced finite difference methods for MFGs; subsequent work treated the planning problem~\cite{achdou2012planning} and established convergence of finite difference approximations~\cite{achdou2013convergence}. Benamou and Carlier~\cite{benamou2015augmented} developed augmented Lagrangian methods for transport optimization and MFGs. Proximal methods have also been applied to stationary MFGs with local couplings~\cite{briceno2018proximal}, and a primal--dual implementation for time-dependent diffusive MFGs combines proximity operators with preconditioned linear solves~\cite{briceno2019implementation}. These methods illustrate how conservation laws and variational structure guide solver design.

On curved surfaces, admissible motion and transport costs depend on intrinsic geometry, as reflected in McCann's theory of optimal maps on Riemannian manifolds~\cite{mccann2001polar}. Computation must account for tangent vector fields, surface differential operators, and geometric approximation. Lavenant et al.~\cite{lavenant2018dynamical} developed dynamical OT on discrete surfaces, while Lavenant~\cite{lavenant2021convergence} established convergence under mesh refinement for quadratic dynamical OT discretizations satisfying a general framework, including schemes on triangulated surfaces. Yu et al.~\cite{yu2023mfg} formulated MFGs on Riemannian manifolds and designed a proximal gradient method for variational problems on triangular meshes.

Surface finite elements provide a natural spatial discretization. Dziuk's method for the Beltrami operator~\cite{dziuk1988beltrami} and the survey by Dziuk and Elliott~\cite{dziuk2013surface} give the underlying framework. Higher-order analysis~\cite{demlow2009higherorder} and adaptive methods~\cite{demlow2007adaptive} distinguish errors in approximating the geometry from errors in approximating the surface PDE. For transport solvers, a further concern is the accuracy of differential quantities passed from the surface PDE solver to the optimization updates.

Proximal algorithms build on proximity operators~\cite{moreau1965proximite}, monotone operator splitting~\cite{lions1979splitting}, and primal--dual formulations~\cite{chambolle2011primaldual}. Writing the objective as the energy $\mathcal{Y}$ plus the indicator $\chi_{\mathcal{C}}$ of the continuity-constraint set separates the energy from conservation. Yu et al.~\cite{yu2024fast} developed a fast proximal gradient method for dynamic MFP whose projection step reduces to an elliptic problem. Papadakis, Peyr\'e, and Oudet~\cite{papadakis2014proximal} applied Douglas--Rachford and primal--dual splitting to dynamic OT. Proximal gradient iterations treat the energy explicitly, whereas Douglas--Rachford splitting evaluates its proximal map. Both require the proximal map of $\chi_{\mathcal{C}}$, namely the projection onto the conservation constraint.

This shared projection is central to the present work. It requires a Poisson equation on the time--space manifold, followed by corrections to $\rho$ and $\mathbf{m}$ using the potential's temporal derivative and surface gradient. Accuracy of the potential alone therefore does not determine accuracy of the updated variables. For smooth solutions on suitable meshes, standard $L^2$ error estimates for linear surface finite elements give one less power of $h$ for the raw gradient than for the potential~\cite{demlow2009higherorder}. These elementwise gradients also differ from the nodal representation of density and momentum.

Gradient recovery addresses this issue while retaining piecewise linear finite elements. The superconvergent patch recovery of Zienkiewicz and Zhu~\cite{zienkiewicz1992recovery} established a widely used approach to derivative reconstruction. The polynomial preserving recovery (PPR) method of Zhang and Naga~\cite{zhang2005recovery} reconstructs derivatives by local polynomial fitting and has superconvergence properties under suitable assumptions. The recovered gradients can also yield asymptotically exact a posteriori error estimators on mildly structured meshes~\cite{naga2004posteriori}. On surfaces, Wei, Chen, and Huang~\cite{wei2010recovery} analyzed superconvergence and gradient recovery for linear finite elements. Dong and Guo's parametric polynomial preserving recovery (PPPR)~\cite{dong2020parametric} reconstructs surface gradients from local approximations of both geometry and function values. These developments motivate nodal derivative recovery within the constraint projection.

Gradient recovery in surface transport has already been demonstrated by the gradient enhanced ADMM method of Dong et al.~\cite{dong2024admm}. Jiang et al.~\cite{jiang2026fdm} further developed an FDM--sFEM scheme for time--space elliptic problems with superconvergence results for recovered gradients. Building on these works, we investigate gradient enhancement at the level of the constraint proximal map. Finite differences in time and surface finite elements in space are combined with temporal PPR and spatial PPPR. The recovered operators enter the right-hand side of the projection Poisson equation and the subsequent density and momentum updates, providing a common construction for different proximal iterations.

We apply this construction to ISTA, FISTA, and Douglas--Rachford splitting. FISTA uses Beck and Teboulle's acceleration~\cite{beck2009fast}, together with backtracking, adaptive restarting, and monotone safeguards~\cite{donoghue2015adaptive,aujol2024parameter,zibetti2018monotone}. The distinction between explicit energy gradients and implicit proximal maps matters near vacuum, where the kinetic energy is singular. The error analysis distinguishes the recovered map from an exact projection and examines how derivative-recovery, data, boundary, and solver errors enter individual updates and finite sequences of iterations. Such discretization estimates are distinct from convergence results for the outer optimization iteration. Experiments examine stepsize dependence, interaction costs on the sphere, transport on the Enzensberger--Stern surface, and refinement behavior.

\begin{samepage}
The main contributions can be summarized as follows:
\begin{enumerate}[before={\interlinepenalty=10000}]
    \item We construct a gradient enhanced approximation of the continuity-constraint proximal map on triangulated surfaces by combining an FDM--sFEM Poisson solver with temporal PPR and spatial PPPR.
    \item We incorporate this shared construction into proximal gradient and Douglas--Rachford iterations, including an accelerated FISTA variant with practical stepsize and monotonicity safeguards.
    \item We derive residual identities and conditional finite-iteration error bounds, and discuss numerical illustrations and recorded refinement data for surface OT and MFP.
\end{enumerate}
\end{samepage}

The remainder of the paper is organized as follows. Section~2 introduces the variational models and notation. Section~3 presents the proximal algorithms. Section~4 develops their discretization and gradient enhancement. Section~5 gives the error analysis, Section~6 discusses the numerical illustrations and their verification limits, and Section~7 concludes the paper.

\endgroup

\section{Preliminaries}
Let $\mathcal{T}=(0,1)$ be the time interval. Let $(\mathcal{M},g)$ be a compact, connected, oriented, $C^3$-smooth two-dimensional surface without boundary, embedded in $\mathbb R^3$ and equipped with the induced metric. Stronger regularity will be stated when required for derivative recovery. We denote by $T\mathcal{M}=\{(\mathbf{x},\mathbf{p})| \mathbf{p}\in T_{\mathbf{x}}\mathcal{M}\}$ the tangent bundle on $\mathcal{M}$.

We consider the dynamic MFP problem with the following formulation:
\begin{equation}
    \begin{aligned}
        &\min_{\rho,\mathbf{m}}&&\mathcal{Y}(\rho,\mathbf{m})\coloneqq\int_{0}^{1}\int_{\mathcal{M}}L\left(\rho(t,x),\mathbf{m}(t,x)\right)\;\mathrm{d}\sigma_{g}\;\mathrm{d}t+\int_{0}^{1}\mathcal{F}\left(\rho (t,\cdot)\right)\mathrm{d}t\\
        &\text{s.t.} && \partial_{t}\rho(t,x) +\nabla_{g}\cdot \mathbf{m}(t,x)=0,\quad\text{ on }\mathcal{T}\times\mathcal{M},\\
        & &&\rho(0,\cdot)=\rho_0,\ \rho(1,\cdot)=\rho_1,\quad\text{ on }\mathcal{M},
    \end{aligned}
    \label{equ:MFP_main_form}
\end{equation}
where $\rho(t,x)$ denotes the population density, $\mathbf{m}=\rho \mathbf{v}$ denotes the momentum with the velocity field $\mathbf{v}(t,\mathbf{x})$ describing the evolution of $\rho$, $L(\rho,\mathbf{m})$ denotes the dynamic cost function and $\mathcal{F}(\rho)$ denotes the interaction cost function. 

\begin{remark}
    The continuity and endpoint constraints in \eqref{equ:MFP_main_form} are affine. We assume that the energy is proper, lower semicontinuous, and convex in the density--momentum pair. Differentiability is required only for the explicit energy steps and on the region where those steps are evaluated. For a local interaction we write $\mathcal F(\rho)=\int_{\mathcal M}F_E(x,\rho(x))\,\mathrm d\sigma_g$. Nonnegativity is imposed through the domain of the energy, rather than included in the affine projection constraint.
\end{remark}

\begin{remark}
    A probability density satisfies $\rho\ge0$, $\rho\in L^\infty(\mathcal T;L^1(\mathcal M))$, and $\int_{\mathcal M}\rho(t,x)\,\mathrm d\sigma_g=1$ for almost every $t\in\mathcal T$.
\end{remark}

\begin{example}
    If we choose the cost functions in \eqref{equ:MFP_main_form} by $\mathcal{F}=0$ and 
    \begin{equation}
        L(\rho,\mathbf{m})= \left\{\begin{aligned}
            &\frac{\left|\mathbf{m}\right|^2}{2\rho}, &&\rho>0,\\
            &0, &&\rho=0,\mathbf{m}=\mathbf{0},\\
            &+\infty, &&\text{otherwise}.
        \end{aligned}\right.
        \nonumber
    \end{equation}

    Then the MFP problem \eqref{equ:MFP_main_form} becomes the dynamical formulation of optimal transport (OT) problem.
\end{example}

Replacing the fixed terminal-density constraint in \eqref{equ:MFP_main_form} by a convex terminal cost $\mathcal G$ gives a variational formulation of potential mean field games (MFGs):
\begin{equation}
    \begin{aligned}
        &\min_{\rho,\mathbf{m}}&&\int_{0}^{1}\int_{\mathcal{M}}L\left(\rho(t,x),\mathbf{m}(t,x)\right)\;\mathrm{d}\sigma_{g}\;\mathrm{d}t+\int_{0}^{1}\mathcal{F}\left(\rho (t,\cdot)\right)\mathrm{d}t+\mathcal{G}\left(\rho(1,\cdot) \right)\\
        &\text{s.t.} && \partial_{t}\rho(t,x) +\nabla_{g}\cdot \mathbf{m}(t,x)=0,&\text{ on }\mathcal{T}\times\mathcal{M},\\
        & &&\rho(0,\cdot)=\rho_0,\ & \text{ on }\mathcal{M},
    \end{aligned}
    \label{equ:MFG_main_form}
\end{equation}
A finite terminal cost generally does not enforce a specified terminal density. The MFG formulation provides context here; the algorithms and numerical illustrations below concern MFP with two prescribed endpoints and its OT special case.

\begin{example}
    An example of the MFG problem \eqref{equ:MFG_main_form} is provided here:
    \begin{equation}
        \begin{aligned}
        &\min_{\rho,\mathbf{m}}&&\int_{0}^{1}\int_{\mathcal{M}}L\left(\rho,\mathbf{m}\right)\mathrm{d}\sigma_{g}\;\mathrm{d}t+\lambda_{E}\int_{0}^{1}\int_{\mathcal{M}}\rho\log \rho\;\mathrm{d}\sigma_g\mathrm{d}t+\lambda_Q\int_{0}^{1}\int_{\mathcal{M}}\rho(t,x)Q(x)\mathrm{d}\sigma_g \mathrm{d}t\\
        & &&+\lambda_G\int_{\mathcal{M}}\rho(1,x) G(x)\mathrm{d}\sigma_g\\
        &\text{s.t.} && \partial_{t}\rho(t,x) +\nabla_{g}\cdot \mathbf{m}(t,x)=0,\quad\text{ on }\mathcal{T}\times\mathcal{M},\\
        & &&\rho(0,\cdot)=\rho_0,\quad \text{ on }\mathcal{M},
    \end{aligned}
    \end{equation}
\end{example}

The endpoint densities are nonnegative and satisfy
\[
\int_{\mathcal M}\rho_0\,\mathrm d\sigma_g
=\int_{\mathcal M}\rho_1\,\mathrm d\sigma_g=1.
\]
For the Hilbert-space projection below we take $\rho_0,\rho_1\in L^2(\mathcal M)$ and set $\mathcal Q=\mathcal T\times\mathcal M$ and
\[
\mathcal H=L^2(\mathcal Q)\times L^2(\mathcal Q;T\mathcal M),\qquad
\|(\rho,\mathbf m)\|_{\mathcal H}^2
=\|\rho\|_{L^2(\mathcal Q)}^2+\|\mathbf m\|_{L^2(\mathcal Q)}^2.
\]
The vector norm is induced by $g$. We occasionally abbreviate a density--momentum pair by $\mathbf z=(\rho,\mathbf m)$, without changing the component notation used in the algorithms.

\section{Algorithm}

For dynamical mean field planning on a surface, we recall the constrained form
\begin{equation}
    \begin{aligned}
    &\min_{\rho,\mathbf{m}}&& \mathcal{Y}(\rho,\mathbf{m})\\
        &\text{s.t.} && \partial_{t}\rho(t,x) +\nabla_{g}\cdot \mathbf{m}(t,x)=0,&\text{ on }\mathcal{T}\times\mathcal{M},\\
        & &&\rho(0,\cdot)=\rho_0,\ \rho(1,\cdot)=\rho_1,& \text{ on }\mathcal{M},
    \end{aligned}
    \label{equ:OT_main_form}
\end{equation}

We use the weak form of the constraint set, which also gives a meaning to the endpoint conditions for $L^2$ fields:
\begin{equation}\label{eq:weak_constraint}
\begin{split}
\mathcal C=\bigl\{(\rho,\mathbf m)\in\mathcal H:
&\ (\rho,\partial_t\psi)_{\mathcal Q}+(\mathbf m,\nabla_g\psi)_{\mathcal Q}
=B(\psi)\quad\forall\psi\in H^1(\mathcal Q)\bigr\},\\
B(\psi)&=(\rho_1,\psi(1))_{\mathcal M}-(\rho_0,\psi(0))_{\mathcal M}.
\end{split}
\end{equation}
For smooth fields this is precisely the continuity equation with the prescribed endpoint densities. The set is closed and affine in $\mathcal H$; the projection construction below also proves that it is nonempty.

Let $v(\rho,\mathbf{m})=\chi_{\mathcal{C}}(\rho,\mathbf{m})$ be the indicator function of $\mathcal{C}$, then the MFP problem \eqref{equ:OT_main_form} can be transformed into the following form
\begin{equation}
    \min_{\rho,\mathbf{m}}\mathcal{Y}(\rho,\mathbf{m})+\chi_{\mathcal{C}}(\rho,\mathbf{m}).
    \label{eq:OT_split_form}
\end{equation}

To deal with $\chi_{\mathcal{C}}$, here we introduce the proximal operator.
\begin{definition}
For a proper, lower semicontinuous, convex function $v(\mathbf{z})$, the proximal operator of $v$ with parameter $\eta>0$ is defined by  
    $$
\operatorname{Prox}_{\eta v}(\mathbf{z}):=\underset{\mathbf{y}}{\operatorname{argmin}}\left\{v(\mathbf{y})+\frac{1}{2 \eta}\|\mathbf{y}-\mathbf{z}\|_2^2\right\} .
$$
\label{def:proximal_operator}
\end{definition}

\subsection{Computation of the constraint proximal map}
To handle $\chi_{\mathcal C}$, we compute its proximal operator, which is the orthogonal projection onto $\mathcal C$:
\begin{equation}\label{eq:projection_step_optimization}
P_{\mathcal C}\mathbf z
=\operatorname{Prox}_{\chi_{\mathcal C}}\mathbf z
=\underset{\mathbf z^\star\in\mathcal C}{\operatorname{argmin}}
\frac12\|\mathbf z^\star-\mathbf z\|_{\mathcal H}^2.
\end{equation}
Its value is independent of the positive proximal parameter. Write $\mathbb G\psi=(\partial_t\psi,\nabla_g\psi)$ and $V=\{\psi\in H^1(\mathcal Q):\int_{\mathcal Q}\psi=0\}$. The variational equation for the potential is
\begin{equation}\label{eq:projection_weak_poisson}
(\mathbb G\varphi,\mathbb G\psi)_{\mathcal H}
=B(\psi)-(\mathbf z,\mathbb G\psi)_{\mathcal H},\qquad \psi\in V.
\end{equation}
\begin{proposition}\label{prop:continuous_projection}
For every $\mathbf z\in\mathcal H$, equation~\eqref{eq:projection_weak_poisson} has a unique solution in $V$, and
\begin{equation}\label{eq:projection_correction}
P_{\mathcal C}(\rho,\mathbf m)
=(\rho+\partial_t\varphi,\mathbf m+\nabla_g\varphi).
\end{equation}
The map $P_{\mathcal C}$ is firmly nonexpansive in $\mathcal H$.
\end{proposition}
\begin{proof}
The trace theorem bounds $B$ on $H^1(\mathcal Q)$. Poincar\'e's inequality makes $\|\mathbb G\psi\|_{\mathcal H}$ an equivalent norm on $V$, so Lax--Milgram gives the potential. Equal endpoint masses imply $B(1)=0$, hence the variational identity also holds for every $\psi\in H^1(\mathcal Q)$. Consequently $\mathbf z^\star=\mathbf z+\mathbb G\varphi$ satisfies~\eqref{eq:weak_constraint}. For any $\mathbf w\in\mathcal C$, subtracting the constraints gives $(\mathbf w-\mathbf z^\star,\mathbb G\varphi)_{\mathcal H}=0$. Expanding the squared distance proves the projection property. Applying this orthogonality to two inputs yields
\[
\|P_{\mathcal C}\mathbf z-P_{\mathcal C}\mathbf w\|_{\mathcal H}^2
=(P_{\mathcal C}\mathbf z-P_{\mathcal C}\mathbf w,\mathbf z-\mathbf w)_{\mathcal H},
\]
which proves firm nonexpansiveness.
\end{proof}

For smooth input and solution, the Lagrangian calculation gives the same update. With the optimization variables written explicitly, the Lagrangian is
\begin{equation}\label{eq:projection_step_Lagrangian}
\mathcal L(\rho^\star,\mathbf m^\star,\varphi)
=\tfrac12\|\rho^\star-\rho\|_{L^2(\mathcal Q)}^2
+\tfrac12\|\mathbf m^\star-\mathbf m\|_{L^2(\mathcal Q)}^2
+(\varphi,\partial_t\rho^\star+\operatorname{div}_g\mathbf m^\star)_{\mathcal Q}.
\end{equation}
Integration by parts yields $\rho^\star=\rho+\partial_t\varphi$ and $\mathbf m^\star=\mathbf m+\nabla_g\varphi$. In particular, the strong Poisson problem is
\begin{equation}\label{eq:Poisson-Equation}
\left\{\begin{aligned}
-\partial_{tt}\varphi-\Delta_g\varphi&=f:=\partial_t\rho+\operatorname{div}_g\mathbf m&&\text{in }\mathcal Q,\\
\partial_t\varphi(0)&=a_0:=\rho_0-\rho(0),\\
\partial_t\varphi(1)&=a_1:=\rho_1-\rho(1),\qquad \int_{\mathcal Q}\varphi=0.
\end{aligned}\right.
\end{equation}
Its compatibility condition is $\int_{\mathcal Q} f=\int_{\mathcal M}(a_0-a_1)$. This follows from equal endpoint masses and the surface divergence theorem. Homogeneous temporal Neumann conditions apply only if the input already has the prescribed endpoints; a general splitting variable need not have them. Equations~\eqref{eq:Poisson-Equation} and~\eqref{eq:projection_correction} constitute the projection computation.

\subsection{\texorpdfstring{Descent of $\mathcal{Y}(\rho,\mathbf{m})$}{Descent of the energy}}
The descent of $\mathcal{Y}(\rho,\mathbf{m})$ can be carried out in two main ways. One approach is to apply a direct gradient step to $\mathcal{Y}$, leading to algorithms such as ISTA, FISTA\cite{yu2023mfg,yu2024fast}. The other is to use the proximal operator $\operatorname{Prox}_{\mathcal{Y}}(\rho,\mathbf{m})$, which gives rise to the proximal splitting methods, including primal-dual schemes, Douglas-Rachford method, and ALG2/ADMM\cite{papadakis2014proximal,lavenant2018dynamical,dong2024admm}.
\begin{enumerate}
    \item \textbf{Gradient Descent}: (ISTA, FISTA)
    $$\left(\rho^{\left(k+\frac{1}{2}\right)}, \mathbf{m}^{\left(k+\frac{1}{2}\right)}\right)=\left({\rho}^{(k)}-\eta^{(k)} \delta_\rho \mathcal{Y}\left({\rho}^{(k)}, {\mathbf{m}}^{(k)}\right), {\mathbf{m}}^{(k)}-\eta^{(k)} \delta_{\mathbf{m}} \mathcal{Y}\left({\rho}^{(k)}, {\mathbf{m}}^{(k)}\right)\right) .$$
    \item \textbf{Proximal (Splitting)}: (Douglas-Rachford, ALG2/ADMM\ldots)
    Specifically, for dynamical optimal transport, $\mathcal{Y}(\rho,\mathbf{m})=\int_0^1\int_{\mathcal{M}}\frac{\mathbf{m}^2}{2\rho}\mathrm{d}\sigma_g\mathrm{d}t$, the proximal of $\mathcal{Y}$ is
    $$
\operatorname{Prox}_{\eta \mathcal{Y}}({\rho},{\mathbf{m}})=\left\{\begin{array}{cl}
\left(\rho^{\star},\frac{\rho^{\star} {\mathbf{m}}}{\rho^{\star}+\eta}\right) & \text { if } \rho^{\star}>0, \\
(0,0) & \text { otherwise. }
\end{array}\right. 
$$
where $\rho^{\star}$ is the largest real root of the 3rd order polynomial equation 
\begin{equation*}
P(X)=(X-{\rho})(X+\eta)^{2}-\frac{\eta}{2}\|{\mathbf{m}}\|^{2}=0 
\end{equation*}

\end{enumerate}

\subsection{ISTA}
To solve a general nonsmooth convex model
\begin{equation}
    \min _{\mathbf{z}} u(\mathbf{z})+v(\mathbf{z}),
    \label{eq:general_convex_model}
\end{equation}
where $u$ is convex with a Lipschitz-continuous gradient and $v$ is proper, lower semicontinuous, convex, and possibly nonsmooth.

With an initial $\mathbf{z}^{(0)}$, one can apply the ISTA method\cite{beck2009fast}

$$
\mathbf{z}^{(k+1)}:=\operatorname{Prox}_{\eta^{(k)} v}\left(\mathbf{z}^{(k)}-\eta^{(k)} \nabla u\left(\mathbf{z}^{(k)}\right)\right) .
$$

Here $\eta^{(k)}>0$ is the step-size satisfying $\eta^{(k)}\leq 1/L(u)$, where $L(u)$ is the Lipschitz constant of $\nabla u$.

In particular, for the splitting form \eqref{eq:OT_split_form} of mean field planning, with initial $\left(\rho^{(0)},\mathbf{m}^{(0)}\right)$, the iteration step can be written as
\begin{equation}
    \left(\rho^{(k+1)},\mathbf{m}^{(k+1)}\right)=\operatorname{Prox}_{\eta^{(k)} \chi_{\mathcal{C}}}\left({\rho}^{(k)}-\eta^{(k)} \delta_\rho \mathcal{Y}\left({\rho}^{(k)}, {\mathbf{m}}^{(k)}\right), {\mathbf{m}}^{(k)}-\eta^{(k)} \delta_{\mathbf{m}} \mathcal{Y}\left({\rho}^{(k)}, {\mathbf{m}}^{(k)}\right)\right),
\end{equation}
where $\delta_{\rho}\mathcal{Y}$ and $\delta_{\mathbf{m}}\mathcal{Y}$ denote the variational derivatives of $\mathcal{Y}$ with respect to $\rho$ and $\mathbf{m}$, respectively.

\begin{remark}\label{rem:positivity}
For optimal transport, the kinetic energy $|\mathbf m|^2/(2\rho)$ is not Lipschitz smooth as $\rho\to0^+$. On a positive-density region its derivatives are
\[
\delta_\rho\mathcal Y=-\frac{|\mathbf m|^2}{2\rho^2}+\partial_\rho F_E,
\qquad \delta_{\mathbf m}\mathcal Y=\frac{\mathbf m}{\rho}.
\]
A lower bound $\rho\ge\varepsilon_\rho>0$ prevents division by zero, but a Lipschitz bound for the gradient also requires control of $\mathbf m$ and of the interaction Hessian. The explicit methods are considered only where their energy and gradient are finite. A post-projection clipping $\rho\leftarrow\max\{\rho,\varepsilon_\rho\}$ generally changes the total mass and the continuity residual. It is therefore not a constraint-preserving projection, and any use of such a heuristic must be accounted for as an additional perturbation. The safeguarded procedure below rejects inadmissible trial points rather than silently clipping them.
\end{remark}

\subsection{FISTA}

FISTA is essentially an accelerated ISTA algorithm\cite{beck2009fast}. At each iteration, it  constructs an extrapolated variable $\overline{\mathbf{z}}^{(k+1)}$ as a linear combination of the updated result $\mathbf{z}^{(k+1)}$ and $\mathbf{z}^{(k)}$ based on Nesterov acceleration, yielding an $\mathcal O(k^{-2})$ objective-gap bound for the standard convex problem with a Lipschitz gradient and exact proximal steps. This statement does not automatically extend to a recovered projection or to all adaptive variants. With initial $\mathbf{z}^{(0)}=\overline{\mathbf{z}}^{(0)}$, the algorithm for solving \eqref{eq:general_convex_model} is summarized in the following formulation:
$$
\left\{\begin{array}{l}
\mathbf{z}^{(k+1)}=\operatorname{Prox}_{\eta^{(k)} v}\left(\overline{\mathbf{z}}^{(k)}-\eta^{(k)} \nabla u\left(\overline{\mathbf{z}}^{(k)}\right)\right), \\
\tau^{(k+1)}=\frac{1}{2}\left(1+\sqrt{1+4\left(\tau^{(k)}\right)^2}\right) , \\
\overline{\mathbf{z}}^{(k+1)}=\mathbf{z}^{(k+1)}+\frac{\tau^{(k)}-1}{\tau^{(k+1)}}\left(\mathbf{z}^{(k+1)}-\mathbf{z}^{(k)}\right),
\end{array}\right.
$$
where $\tau^{(k)}$ is the acceleration parameter with initial value $\tau^{(0)}=1$.

For the splitting form of mean field planning \eqref{eq:OT_split_form}, with initial $\left(\rho^{(0)},\mathbf{m}^{(0)}\right)=\left(\overline{\rho}^{(0)},\overline{\mathbf{m}}^{(0)}\right)$ , the iteration form can be rewritten as
\begin{equation}
    \left\{\begin{array}{l}
\left(\rho^{(k+1)},\mathbf{m}^{(k+1)}\right)=\operatorname{Prox}_{\eta^{(k)} v}\left(\overline{\rho}^{(k)}-\eta^{(k)} \delta_\rho \mathcal{Y}\left(\overline{\rho}^{(k)}, \overline{\mathbf{m}}^{(k)}\right), \overline{\mathbf{m}}^{(k)}-\eta^{(k)} \delta_{\mathbf{m}} \mathcal{Y}\left(\overline{\rho}^{(k)}, \overline{\mathbf{m}}^{(k)}\right)\right) ,\\
\tau^{(k+1)}=\frac{1}{2}\left(1+\sqrt{1+4\left(\tau^{(k)}\right)^2}\right) , \\
\left(\overline{\rho}^{(k+1)},\overline{\mathbf{m}}^{(k+1)}\right)=\left(\rho^{(k+1)},\mathbf{m}^{(k+1)}\right)+\frac{\tau^{(k)}-1}{\tau^{(k+1)}}\left(\left(\rho^{(k+1)},\mathbf{m}^{(k+1)}\right)-\left(\rho^{(k)},\mathbf{m}^{(k)}\right)\right).
\end{array}\right.
\label{eq:fista-update_MFP}
\end{equation}

FISTA exhibits a strong dependence on both the step size and the acceleration parameter $\tau^{(k)}$. Therefore, inspired by the parameter-free FISTA framework\cite{aujol2024parameter}, we adopt the adaptive backtracking technique and adaptive restarting techniques in FISTA.

\subsubsection{Backtracking and adaptive restarting}
For a given extrapolated point $\overline{\mathbf{z}}^{(k)}$,
the FISTA update \eqref{eq:fista-update_MFP} also requires a step size $\eta^{(k)}$ satisfying 
$\eta^{(k)}\leq 1/L(\mathcal{Y})$, where $L(\mathcal{Y})$ is the Lipschitz constant of $\nabla\mathcal{Y}$. 
Instead of prescribing the Lipschitz constant in advance, the backtracking strategy \cite{beck2009fast} searches for 
an admissible step size $\eta^{(k)}$ such that the following descent inequality holds:
\begin{equation}
    \begin{aligned}
        &\mathcal{Y}\left(\rho^{(k+1)},\mathbf{m}^{(k+1)}\right)\leq 
\mathcal{Y}\left(\overline{\rho}^{(k)},\overline{\mathbf{m}}^{(k)}\right)
+\frac{1}{2\eta^{(k)}}
\left\|\left(\rho^{(k+1)},\mathbf{m}^{(k+1)}\right)-\left(\overline{\rho}^{(k)},\overline{\mathbf{m}}^{(k)}\right)\right\|^2 \\
&+\left\langle \left(\delta_{\rho}\mathcal{Y}\left(\overline{\rho}^{(k)},\overline{\mathbf{m}}^{(k)}\right),\delta_{\mathbf{m}}\mathcal{Y}\left(\overline{\rho}^{(k)},\overline{\mathbf{m}}^{(k)}\right)\right),
\left({\rho}^{(k+1)}-\overline{\rho}^{(k)},{\mathbf{m}}^{(k+1)}-\overline{\mathbf{m}}^{(k)}\right)
\right\rangle.
    \end{aligned}
    \label{eq:backtracking_condition}
\end{equation}

If \eqref{eq:backtracking_condition} is not satisfied, 
 $\eta^{(k)}$ is reduced by
 $\eta^{(k)}\leftarrow \mu \eta^{(k)}$ with $\mu\in(0,1)$,
 and the update \eqref{eq:fista-update_MFP} is recomputed. The finite trial limit and failure handling used for recovered steps are specified in Section~\ref{sec:discretization}.

 In addition, we adopt an adaptive restart strategy for
the FISTA acceleration parameter\cite{donoghue2015adaptive,aujol2024parameter}. The purpose of restarting is to remove the
inertial effect when the extrapolation no longer provides a reliable descent
direction. Once a restart is triggered, both the extrapolated point and the acceleration state are reset: set $\overline{\mathbf z}^{(k)}=\mathbf z^{(k)}$ and $\tau^{(k)}=1$, and recompute the trial step. After acceptance the extrapolation coefficient $(\tau^{(k)}-1)/\tau^{(k+1)}$ is zero.

For the safeguarded variant, given $\left(\rho^{(k+1)},\mathbf{m}^{(k+1)}\right)$ computed by \eqref{eq:backtracking_condition}, we test if it satisfies the following monotonicity condition:
\begin{equation}
    \mathcal{Y}\left(\rho^{(k+1)},\mathbf{m}^{(k+1)}\right)
    \leq
    \mathcal{Y}\left({\rho}^{(k)},{\mathbf{m}}^{(k)}\right)
    +
    \varepsilon_{\mathrm{obj}}
    \max\left\{1,\left|\mathcal{Y}\left({\rho}^{(k)},{\mathbf{m}}^{(k)}\right)\right|\right\},
    \label{eq:monotone_condition}
\end{equation}
where $\varepsilon_{\mathrm{obj}}$ is a small numerical tolerance. If the condition fails, the trial is discarded and a new trial is computed from the current accepted point with the acceleration state reset. A gradient-based restart for a composite problem instead uses
\begin{equation}\label{eq:gradient_restart}
\left\langle\overline{\mathbf z}^{(k)}-\mathbf z^{(k+1)},
\mathbf z^{(k+1)}-\mathbf z^{(k)}\right\rangle>0,
\end{equation}
with the same inner product as the proximal step~\cite{donoghue2015adaptive}. It tests the generalized gradient against the successive-iterate displacement. The inner product of the energy gradient with the trial displacement from the extrapolated point is not this restart test. For a recovered update these rules are safeguards, not a proof of monotonicity of $\mathcal Y+\chi_{\mathcal C}$.

\subsubsection{Monotone FISTA}

Although the FISTA extrapolation improves the convergence rate in many convex
optimization problems, the extrapolated point may lead to an increase of the
objective value. This issue becomes more pronounced for the dynamical optimal transport
energy. To improve the robustness of the accelerated scheme,
we incorporate a monotone safeguard following the idea of monotone FISTA
(MFISTA) \cite{beck2009fast,zibetti2018monotone}.

With initial $\left(\rho^{(0)},\mathbf{m}^{(0)}\right)=\left(\overline{\rho}^{(0)},\overline{\mathbf{m}}^{(0)}\right)$, 
we first compute the proximal gradient candidate
\begin{equation}
    \left(\rho^{(k+1)}_{\text{C}},\mathbf{m}_{\text{C}}^{(k+1)}\right)=\operatorname{Prox}_{\eta^{(k)} v}\left(\overline{\rho}^{(k)}-\eta^{(k)} \delta_\rho \mathcal{Y}\left(\overline{\rho}^{(k)}, \overline{\mathbf{m}}^{(k)}\right), \overline{\mathbf{m}}^{(k)}-\eta^{(k)} \delta_{\mathbf{m}} \mathcal{Y}\left(\overline{\rho}^{(k)}, \overline{\mathbf{m}}^{(k)}\right)\right),
    \label{eq:mfista_candidate}
\end{equation}

 Different from the standard FISTA update,
MFISTA does not necessarily accept $(\rho_{\mathrm C}^{(k+1)},\mathbf m_{\mathrm C}^{(k+1)})$ as the next
iterate. Instead, it compares the objective value of the newly computed proximal
point with that of the previously accepted iterate and defines
\begin{equation}
    \left(\rho^{(k+1)},\mathbf{m}^{(k+1)}\right)
    =
    \underset{(\rho,\mathbf{m})\in
    \left\{
    \left(\rho^{(k+1)}_{\text{C}},\mathbf{m}_{\text{C}}^{(k+1)}\right),\left({\rho}^{(k)}, {\mathbf{m}}^{(k)}\right)
    \right\}}{\operatorname{argmin}} 
    \mathcal{Y}(\rho,\mathbf{m}),
    \label{eq:mfista_monotone_selection}
\end{equation}
where the monotonicity condition \eqref{eq:monotone_condition} can be used as
a criterion for comparing the objective values at
$\left(\rho_{\mathrm{C}}^{(k+1)},\mathbf{m}_{\mathrm{C}}^{(k+1)}\right)$ and
$\left(\rho^{(k)},\mathbf{m}^{(k)}\right)$ in practice.

After updating the acceleration parameter by $\tau^{(k+1)}=\frac{1}{2}\left(1+\sqrt{1+4\left(\tau^{(k)}\right)^2}\right)$, the next extrapolated point is then constructed by
\begin{equation}
    \begin{aligned}
        \left(\overline{\rho}^{(k+1)},\overline{\mathbf{m}}^{(k+1)}\right)
    =&
    \left(\rho^{(k+1)},\mathbf{m}^{(k+1)}\right)
    +
    \frac{\tau^{(k)}}{\tau^{(k+1)}}
    \left(
    \left(\rho^{(k+1)}_{\text{C}},\mathbf{m}_{\text{C}}^{(k+1)}\right)
    -
    \left(\rho^{(k+1)},\mathbf{m}^{(k+1)}\right)
    \right)\\
    &+
    \frac{\tau^{(k)}-1}{\tau^{(k+1)}}
    \left(
    \left(\rho^{(k+1)},\mathbf{m}^{(k+1)}\right)
    -
    \left(\rho^{(k)},\mathbf{m}^{(k)}\right)
    \right).
    \end{aligned}
    \label{eq:mfista_extrapolation}
\end{equation}

When the proximal candidate is accepted, i.e.,
$(\rho^{(k+1)},\mathbf m^{(k+1)})=(\rho_{\mathrm C}^{(k+1)},\mathbf m_{\mathrm C}^{(k+1)})$, the first correction term in
\eqref{eq:mfista_extrapolation} vanishes, and the update reduces to the standard
FISTA extrapolation \eqref{eq:fista-update_MFP}.
For this energy-only comparison to be the standard exact MFISTA comparison for $\mathcal Y+\chi_{\mathcal C}$, the initial accepted point and both candidates must belong to $\mathcal C\cap\operatorname{dom}\mathcal Y$. With recovered updates, the comparison is only an energy safeguard and must be accompanied by a separate constraint check.

\subsection{Proximal Splitting}
The other approach is to use the proximal operator of $\mathcal{Y}$ itself, 
which leads to proximal splitting methods such as Douglas-Rachford splitting, 
primal-dual methods, and ALG2/ADMM-type algorithms 
\cite{papadakis2014proximal,lavenant2018dynamical,dong2024admm}.

The proximal splitting approach treats the two terms in 
\eqref{eq:OT_split_form} in a symmetric implicit manner. 
Different from ISTA or FISTA, where the functional 
$\mathcal{Y}$ is linearized by an explicit gradient step, 
Douglas-Rachford splitting applies the proximal operators of both 
$\mathcal{Y}$ and $\chi_{\mathcal{C}}$. This is particularly useful for 
dynamical optimal transport, since the proximal operator of the kinetic 
energy density $\frac{|\mathbf{m}|^2}{2\rho}$ can be computed pointwise, 
while the proximal operator of $\chi_{\mathcal{C}}$ is exactly the projection 
onto the continuity-equation constraint set.

Let $\left(\overline{\rho}^{(k)},\overline{\mathbf{m}}^{(k)}\right)$ be the splitting variables, with initial $\left(\rho^{(0)},\mathbf{m}^{(0)}\right)$ and $\left(\overline{\rho}^{(0)},\overline{\mathbf{m}}^{(0)}\right)$, the Douglas-Rachford 
iteration for \eqref{eq:OT_split_form} reads
\begin{equation}
    \left\{
    \begin{aligned}
        \left(\overline{\rho}^{(k+1)},\overline{\mathbf{m}}^{(k+1)}\right)
        &=\left(\overline{\rho}^{(k)},\overline{\mathbf{m}}^{(k)}\right)
        +\alpha\left(\operatorname{Prox}_{\gamma \mathcal{Y}}
        \left(2\rho^{(k)}-\overline{\rho}^{(k)},2\mathbf{m}^{(k)}-\overline{\mathbf{m}}^{(k)}\right)
        -\left(\mathbf{\rho}^{(k)},\mathbf{m}^{(k)}\right)\right),\\
        \left(\rho^{(k+1)},\mathbf{m}^{(k+1)}\right)
        &=\operatorname{Prox}_{\gamma \chi_{\mathcal{C}}}
        \left(\overline{\rho}^{(k+1)},\overline{\mathbf{m}}^{(k+1)}\right),
    \end{aligned}
    \right.
    \label{eq:DR_iteration_general}
\end{equation}
where $\gamma>0$ is a fixed proximal parameter and $\alpha\in(0,2)$ is the relaxation 
parameter. The initialization is constrained by
$(\rho^{(0)},\mathbf m^{(0)})=P_{\mathcal C}(\overline\rho^{(0)},\overline{\mathbf m}^{(0)})$.
The barred variables are auxiliary splitting variables; the unbarred sequence is the projected sequence.

Note that in \eqref{eq:DR_iteration_general}, the proximal operator of $\mathcal{Y}$ can be evaluated according to Definition~\ref{def:proximal_operator}, leading to an optimization problem, which can be handled either by numerical schemes, such as Newton's method, or, in certain cases, by deriving a closed-form analytical solution. For example, the proximal operator of $\mathcal{Y}$ for dynamical optimal transport admits the following form\cite{papadakis2014proximal}.

\begin{proposition}
    For the classical dynamical optimal transport cost
\[
    \mathcal{Y}(\rho,\mathbf{m})
    =
    \int_0^1\int_{\mathcal{M}}
    \frac{|\mathbf{m}|^2}{2\rho}\,
    \mathrm{d}\sigma_g\,\mathrm{d}t,
\]
the proximal operator 
$\operatorname{Prox}_{\gamma\mathcal{Y}}$ is computed pointwise as
\begin{equation}
\operatorname{Prox}_{\gamma \mathcal{Y}}({\rho},{\mathbf{m}})
=
\begin{cases}
\left(
\rho^{\star},
\dfrac{\rho^{\star}\mathbf{m}}{\rho^{\star}+\gamma}
\right),
& \rho^{\star}>0,\\[1.0em]
(0,0), & \text{otherwise},
\end{cases}
\label{eq:prox_Y_DR}
\end{equation}
where $\rho^{\star}$ is the largest real root of the cubic equation
\begin{equation}
    P(X)
    =
    (X-\rho)(X+\gamma)^2
    -\frac{\gamma}{2}\|\mathbf{m}\|^2
    =
    0.
    \label{eq:prox_Y_cubic_DR}
\end{equation}
\end{proposition}
\begin{proof}
For a fixed nonnegative output density $X$, minimizing the local proximal objective over momentum gives $\mathbf m^\star=X\mathbf m/(X+\gamma)$. The remaining scalar objective, up to a constant independent of $X$, is
\[
\tfrac12(X-\rho)^2+\frac{\gamma|\mathbf m|^2}{2(X+\gamma)},\qquad X\ge0.
\]
It is strictly convex. Its derivative is $X-\rho-\gamma|\mathbf m|^2/(2(X+\gamma)^2)$. If the derivative at zero is nonnegative, the minimizer is zero. Otherwise it has a unique positive zero, which is the positive, and hence largest, real root of~\eqref{eq:prox_Y_cubic_DR}. This gives~\eqref{eq:prox_Y_DR}, including vacuum~\cite{papadakis2014proximal}.
\end{proof}

Consequently, the Douglas-Rachford method alternates between two implicit 
operations: a local proximal update for the transportation cost 
$\mathcal{Y}$, and a global projection onto the constraint set 
$\mathcal{C}$. The former is reduced to solving independent cubic equations 
at each time--space quadrature node when the discrete energy and metric use matching diagonal weights, while the latter is reduced to solving a 
Poisson equation on the time-space manifold. In this sense, Douglas-Rachford 
splitting separates the nonlinear transport cost and the linear continuity 
constraint into two computationally tractable subproblems.

For an exact projection, standard Douglas--Rachford theory applies under its usual solvability assumptions~\cite{lions1979splitting,papadakis2014proximal}. An intermediate projected density need not be nonnegative; the kinetic proximal map has the appropriate extended-value domain. At termination, proximity of the two proximal outputs and the constraint residual must both be checked. A general nonlocal interaction does not necessarily admit a pointwise proximal calculation.

\section{Discretization}\label{sec:discretization}

Consider a shape-regular triangulated approximation $\mathcal M_h$ of $\mathcal M$, with mesh size $h$. We denote the triangulation by $\mathcal T_h$ and its vertices by $\mathcal V_h=\{V_i\}_{i=1}^{N_v}$. The piecewise linear finite element space is
\[
S_h=\{v_h\in C^0(\mathcal M_h):v_h|_T\in\mathbb P^1(T),\ T\in\mathcal T_h\}.
\]
We divide $[0,1]$ into $N_t\ge2$ uniform subintervals, with $t_j=j\tau$ and $\tau=1/N_t$. Here $\tau$ without a superscript is the time step; $\tau^{(k)}$ continues to denote the FISTA acceleration parameter. Let $S_\tau$ be the continuous piecewise linear temporal space with nodal basis $\{\phi_j\}_{j=0}^{N_t}$, and set $Q_h=S_\tau\otimes S_h$. Thus $u_h(t,x)=\sum_{j=0}^{N_t}u_{h,j}(x)\phi_j(t)$.

\subsection{Inner products, energy, and temporal boundary conditions}
Let $\{\psi_i\}_{i=1}^{N_v}$ be the nodal basis of $S_h$. Write
\[
M_{i\ell}=(\psi_i,\psi_\ell)_{\mathcal M_h},\quad
K_{i\ell}=(\nabla_{g_h}\psi_i,\nabla_{g_h}\psi_\ell)_{\mathcal M_h},\quad
A_i=\int_{\mathcal M_h}\psi_i\,\mathrm d\sigma_h>0,
\]
and use the trapezoidal time weights $w_0=w_{N_t}=\tau/2$ and $w_j=\tau$ otherwise. The Poisson equation uses the consistent spatial mass matrix $M$. The optimization metric and local energy use matching lumped weights:
\begin{equation}\label{eq:discrete_metric_energy}
\begin{split}
\|\mathbf z_h\|_h^2
&=\sum_{j=0}^{N_t}\sum_{i=1}^{N_v}w_jA_i
\bigl(|\rho_{h,ji}|^2+|\mathbf m_{h,ji}|^2\bigr),\\
\mathcal Y_h(\rho_h,\mathbf m_h)
&=\sum_{j=0}^{N_t}\sum_{i=1}^{N_v}w_jA_i
\bigl[L(\rho_{h,ji},\mathbf m_{h,ji})+F_E(V_i,\rho_{h,ji})\bigr].
\end{split}
\end{equation}
The matrix of this metric is denoted by $W_h$. In particular, the variational derivatives $\delta_{\rho_h}\mathcal Y_h$ and $\delta_{\mathbf m_h}\mathcal Y_h$ are gradients in this metric, not unweighted derivatives of the coefficient vector. For a nonlocal $\mathcal F$, its discretization replaces the last sum and its gradient or proximal map must be computed accordingly. The local OT proximal formula in Section~3 is nodewise precisely because the energy and metric in~\eqref{eq:discrete_metric_energy} have the same weights.

For continuous endpoint data, the endpoint interpolants are normalized separately to unit surface mass:
\[
\rho_{\ell,h}=\frac{I_hT_h\rho_\ell}{\int_{\mathcal M_h}I_hT_h\rho_\ell\,\mathrm d\sigma_h},\qquad \ell=0,1,
\]
provided the denominators are positive. Here $T_h$ transfers a function from $\mathcal M$ to $\mathcal M_h$ by the geometric correspondence, and $I_h$ is nodal interpolation. This normalization is part of the endpoint approximation; its error is retained in the analysis. Nonsmooth endpoint data require a specified nodal or cell-to-node representation in place of $I_h$.

For a nodal sequence, define $(D_tu)_0=(u_1-u_0)/\tau$ and $(D_tu)_j=(u_j-u_{j-1})/\tau$ for $1\le j\le N_t$. The homogeneous-Neumann matrix for the second derivative is
\begin{equation}\label{eq:second-order_difference}
(D_{tt}u)_j=\begin{cases}
2(u_1-u_0)/\tau^2,&j=0,\\
(u_{j-1}-2u_j+u_{j+1})/\tau^2,&1\le j<N_t,\\
2(u_{N_t-1}-u_{N_t})/\tau^2,&j=N_t.
\end{cases}
\end{equation}
Thus $D_{tt}$ approximates $\partial_{tt}$, and $-D_{tt}$ is nonnegative in the trapezoidal metric. For input $\mathbf z_h=(\rho_h,\mathbf m_h)$ set
\begin{equation}\label{eq:discrete_boundary_load}
\begin{gathered}
a_{0,h}=\rho_{0,h}-\rho_{h,0},\qquad
a_{1,h}=\rho_{1,h}-\rho_{h,N_t},\\
b_{h,j}(\mathbf z_h)=\begin{cases}
-2a_{0,h}/\tau,&j=0,\\
0,&1\le j<N_t,\\
2a_{1,h}/\tau,&j=N_t.
\end{cases}
\end{gathered}
\end{equation}
Indeed, eliminating ghost values from $(\varphi_1-\varphi_{-1})/(2\tau)=a_{0,h}$ and $(\varphi_{N_t+1}-\varphi_{N_t-1})/(2\tau)=a_{1,h}$ gives exactly these right-hand-side corrections for $-\partial_{tt}\varphi$. In particular, $b_h=0$ when the input endpoints are already fixed.

\subsection{Gradient-enhanced approximation of the projection}
The temporal polynomial preserving recovery (PPR) operator $G_\tau:S_\tau\to S_\tau$ is obtained by quadratic fitting~\cite{zhang2005recovery,dong2024admm,jiang2026fdm}. On this uniform grid its nodal values are
\begin{equation}\label{eq:PPR_stencil}
(G_\tau u)_j=\begin{cases}
(-3u_0+4u_1-u_2)/(2\tau),&j=0,\\
(u_{j+1}-u_{j-1})/(2\tau),&1\le j<N_t,\\
(3u_{N_t}-4u_{N_t-1}+u_{N_t-2})/(2\tau),&j=N_t.
\end{cases}
\end{equation}
Writing $u_\tau=U_\tau^\top\phi_t$, where $\phi_t=(\phi_0,\ldots,\phi_{N_t})^\top$, gives
\begin{equation}\label{eq:PPR_recovery_representation}
G_\tau u_\tau=(B_tU_\tau)^\top\phi_t.
\end{equation}

The spatial parametric polynomial preserving recovery (PPPR) operator reconstructs both geometry and function values~\cite{dong2020parametric}. At a vertex, project a connected vertex patch onto a local plane and fit a quadratic parametrization $r_i:\mathbb R^2\to\mathbb R^3$ and a quadratic scalar polynomial $q_i$. If $J_i=Dr_i(0)$ has rank two, the recovered gradient is
\begin{equation}\label{eq:PPPR_construction}
(G_hu_h)(V_i)=J_i(J_i^\top J_i)^{-1}\nabla q_i(0).
\end{equation}
One concrete construction takes the initial plane orthogonal to the normalized area-weighted average of incident face normals, scales local coordinates by the patch diameter, and uses unweighted least squares. Starting from the one-ring patch, complete vertex rings are added until the quadratic design matrix has rank six. QR factorization can be used for the fits. A degenerate plane or a rank-deficient fitted parametrization requires repairing or rejecting the patch; it must not be replaced by an unspecified inverse. Uniform conditioning and the geometric hypotheses of the recovery theory are additional requirements for asymptotic estimates.

The fitted normal $\mathbf n_i$ is the unit normal to the range of $J_i$. Nodal momenta are represented in an orthonormal basis of the plane $\mathbf n_i^\perp$; we denote the resulting density--momentum space by $X_h$. Tangency here is with respect to the reconstructed nodal geometry, rather than simultaneous tangency to every incident flat face. The matrices of the recovered gradient satisfy
\begin{equation}\label{eq:PPPR_recovery_representation}
G_hu_h=\bigl((B_xU_h)^\top\psi,(B_yU_h)^\top\psi,(B_zU_h)^\top\psi\bigr),
\qquad \psi=(\psi_1,\ldots,\psi_{N_v})^\top.
\end{equation}
The recovered divergence used below is the componentwise trace
$G_h\cdot\mathbf m_h=B_xm_x+B_ym_y+B_zm_z$ at the nodes, followed by interpolation. This definition does not assert that it is the negative adjoint of $G_h$ in the optimization metric.

For $u_h,v_h\in Q_h$ define
\begin{equation}\label{eq:poisson_bilinear}
\begin{split}
\mathfrak a_h(u_h,v_h)
={}&\sum_{j=0}^{N_t}w_j\bigl[-(D_{tt}u_{h,j},v_{h,j})_{\mathcal M_h}
+(\nabla_{g_h}u_{h,j},\nabla_{g_h}v_{h,j})_{\mathcal M_h}\bigr]\\
={}&\sum_{j=1}^{N_t}\tau\left(\frac{u_{h,j}-u_{h,j-1}}{\tau},
\frac{v_{h,j}-v_{h,j-1}}{\tau}\right)_{\mathcal M_h}
+\sum_{j=0}^{N_t}w_j(\nabla_{g_h}u_{h,j},\nabla_{g_h}v_{h,j})_{\mathcal M_h}.
\end{split}
\end{equation}
The second equality follows by summation by parts. Hence this form is positive definite on
\[
V_h=\left\{v_h\in Q_h:\sum_jw_j(v_{h,j},1)_{\mathcal M_h}=0\right\}.
\]
The ordinary FDM--sFEM right-hand side contains $D_t\rho_h+\operatorname{div}_{g_h}\mathbf m_h+b_h$. Its gradient-enhanced counterpart is
\begin{equation}\label{eq:recovered_load}
\ell_h(\mathbf z_h;v_h)=\sum_{j=0}^{N_t}w_j
\bigl(G_\tau\rho_{h,j}+G_h\cdot\mathbf m_{h,j}+b_{h,j}(\mathbf z_h),v_{h,j}\bigr)_{\mathcal M_h}.
\end{equation}
This functional is affine in $\mathbf z_h$. We compute $\varphi_h\in V_h$ from
\begin{equation}\label{eq:Poisson-gradient_enhanced}
\mathfrak a_h(\varphi_h,v_h)=\ell_h(\mathbf z_h;v_h),\qquad v_h\in V_h.
\end{equation}
Equivalently, the full-space right-hand side is centered by subtracting
\begin{equation}\label{eq:compatibility_defect}
c_h(\mathbf z_h)=\frac{\ell_h(\mathbf z_h;1)}{|\mathcal M_h|}.
\end{equation}
In coefficient form one may solve the resulting system with a zero-mean constraint or use its inverse on $V_h$. Centering ensures solvability, but does not enforce the discrete continuity equation. The discarded constant $c_h$ must therefore be monitored as a compatibility defect. The Poisson matrix and its factorization can be reused across iterations and trial step sizes on a fixed mesh.

Set $\mathbb G_h\varphi_h=(G_\tau\varphi_h,G_h\varphi_h)$. Let $\mathcal R_h$ reset the two density endpoint layers to $\rho_{0,h},\rho_{1,h}$ and project nodal momenta by $\Pi_i=I-\mathbf n_i\mathbf n_i^\top$. The recovered update is
\begin{equation}\label{eq:GE_projection_map}
\widetilde P_h\mathbf z_h
=\mathcal R_h(\mathbf z_h+\mathbb G_h\varphi_h).
\end{equation}
Before this endpoint and tangency enforcement, its components are exactly the original gradient-enhanced corrections
\[
\rho_{h,j}^{\star}=\rho_{h,j}+G_\tau\varphi_{h,j},\qquad
\mathbf m_{h,j}^{\star}=\mathbf m_{h,j}+G_h\varphi_{h,j}.
\]
The endpoint reset is necessary because the recovered boundary derivative need not equal the prescribed Neumann datum. Its contribution is retained in the error bounds. No density clipping or time-slice mass renormalization is included in~\eqref{eq:GE_projection_map}.

\subsection{Constraint residual and exact discrete projection}
To specify what is being checked, define a weak discrete continuity functional using the actual piecewise linear fields:
\begin{equation}\label{eq:discrete_weak_constraint}
\mathcal B_h(\mathbf z_h,v_h)=\int_0^1\int_{\mathcal M_h}
\bigl(\rho_h\partial_tv_h+\mathbf m_h\cdot\nabla_{g_h}v_h\bigr)\,\mathrm d\sigma_h\,\mathrm dt,
\quad
B_h(v_h)=(\rho_{1,h},v_{h,N_t})_{\mathcal M_h}-(\rho_{0,h},v_{h,0})_{\mathcal M_h}.
\end{equation}
Let $A_h\mathbf z_h=\mathbf b_h$ collect $\mathcal B_h(\mathbf z_h,\phi_j\psi_i)=B_h(\phi_j\psi_i)$ for all tensor-product basis functions, together with the endpoint coefficient equations. These are exact integrals of the finite element functions. Thus $\mathbf b_h$ is a constraint vector, distinct from the ghost-point load $b_{h,j}$. The residual uses this specified basis and the Euclidean norm of its coefficient vector.

If $\mathcal C_h=\{\mathbf z_h\in X_h:A_h\mathbf z_h=\mathbf b_h\}$ is nonempty, its exact $W_h$-orthogonal projection satisfies
\begin{equation}\label{eq:exact_discrete_projection}
P_{\mathcal C_h}^{W_h}\mathbf z_h=\mathbf z_h-W_h^{-1}A_h^\top\lambda_h,
\qquad
A_hW_h^{-1}A_h^\top\lambda_h=A_h\mathbf z_h-\mathbf b_h.
\end{equation}
Here matrices act in scalar and two-component nodal tangent coordinates; redundant constraints are handled by a generalized inverse. Feasibility of $\mathcal C_h$ is a separate linear-algebra condition. Equation~\eqref{eq:exact_discrete_projection} is a comparison identity, not the system used by~\eqref{eq:Poisson-gradient_enhanced}. In particular, replacing derivatives by recovered derivatives does not identify the FDM--sFEM matrix with this normal operator. We therefore call $\widetilde P_h$ an approximation of the continuous projection, not an exact discrete proximal map.

\subsection{Gradient-enhanced algorithms}
In the following algorithms, $\mathbf z_h=(\rho_h,\mathbf m_h)$ and $\nabla_h\mathcal Y_h=(\delta_{\rho_h}\mathcal Y_h,\delta_{\mathbf m_h}\mathcal Y_h)$. A suitable positive starting density is the linear interpolation $(1-t_j)\rho_{0,h}+t_j\rho_{1,h}$ when both endpoints are strictly positive, with zero momentum. This satisfies the endpoint and mass conditions but need not satisfy continuity. The explicit methods require a positive initial density at every quadrature node; zero endpoint data cannot be replaced by positive data without changing the problem.

A common trial routine $\operatorname{Trial}_h(\mathbf y_h;\mathbf z_h,\eta)$ computes
\begin{equation}\label{eq:trial_step}
\mathbf z_h^{1/2}=\mathbf y_h-\eta\nabla_h\mathcal Y_h(\mathbf y_h),\qquad
\mathbf z_h^{+}=\widetilde P_h\mathbf z_h^{1/2}.
\end{equation}
The routine requires $\mathbf y_h$ to have finite energy and gradient and density at least $\varepsilon_\rho>0$. It accepts a trial only if the candidate also satisfies this bound and the discrete version of~\eqref{eq:backtracking_condition} holds in $\|\cdot\|_h$. If a monotone safeguard is requested, it additionally checks~\eqref{eq:monotone_condition} against $\mathbf z_h$. Otherwise it replaces $\eta$ by $\mu\eta$, $0<\mu<1$, and recomputes both parts of~\eqref{eq:trial_step}. A prescribed minimum step $\eta_{\min}>0$ and maximum number of trials $N_{\rm bt}$ make this a finite procedure. An invalid base point or exhausted search returns failure, rather than a clipped density or an unaccepted step. Poisson solves use a prescribed residual tolerance.

\begin{algorithm}[htbp]
\caption{Gradient-enhanced ISTA for MFP}\label{alg:GradientEnhanced_ISTA}
\begin{algorithmic}[1]
\Require Endpoints, $\mathcal Y_h$, positive $\mathbf z_h^{(0)}$, trial and stopping parameters, $K_{\max}$.
\For{$k=0,\ldots,K_{\max}-1$}
\State Compute $(\mathbf z_h^{(k+1)},\eta^{(k)})$ by $\operatorname{Trial}_h(\mathbf z_h^{(k)};\mathbf z_h^{(k)},\eta_{\rm init}^{(k)})$.
\State If the trial fails, return the last accepted iterate and a failure status.
\State Check the stopping diagnostics below and return if satisfied.
\EndFor
\State Return the last accepted iterate with an iteration-limit status.
\end{algorithmic}
\end{algorithm}
The correction in this algorithm is added to the half-step pair, not to the old iterate. Without the safeguards, its basic update is $\mathbf z_h^{(k+1)}=\widetilde P_h(\mathbf z_h^{(k)}-\eta^{(k)}\nabla_h\mathcal Y_h(\mathbf z_h^{(k)}))$.

\begin{algorithm}[htbp]
\caption{Gradient-enhanced FISTA with backtracking and restart}\label{alg:Gradient_enhanced_FISTA}
\begin{algorithmic}[1]
\Require The data of Algorithm~\ref{alg:GradientEnhanced_ISTA}; choice of monotone and gradient-restart safeguards.
\State Set $\overline{\mathbf z}_h^{(0)}=\mathbf z_h^{(0)}$ and $\tau^{(0)}=1$.
\For{$k=0,\ldots,K_{\max}-1$}
\State Try $\operatorname{Trial}_h(\overline{\mathbf z}_h^{(k)};\mathbf z_h^{(k)},\eta_{\rm init}^{(k)})$.
\If{the trial fails, or the requested test~\eqref{eq:gradient_restart} triggers}
\State Reset $\overline{\mathbf z}_h^{(k)}=\mathbf z_h^{(k)}$ and $\tau^{(k)}=1$.
\State Recompute the trial from this base point; on failure return the last accepted iterate and a failure status.
\EndIf
\State Accept $\mathbf z_h^{(k+1)}$ and its step $\eta^{(k)}$.
\State Set $\tau^{(k+1)}=(1+\sqrt{1+4(\tau^{(k)})^2})/2$ and $\omega^{(k)}=(\tau^{(k)}-1)/\tau^{(k+1)}$.
\State Set $\overline{\mathbf z}_h^{(k+1)}=\mathbf z_h^{(k+1)}+\omega^{(k)}(\mathbf z_h^{(k+1)}-\mathbf z_h^{(k)})$.
\State Check the stopping diagnostics and return if satisfied.
\EndFor
\State Return the last accepted iterate with an iteration-limit status.
\end{algorithmic}
\end{algorithm}
The restart inner product uses $W_h$. MFISTA is a distinct option: replace candidate acceptance and extrapolation by~\eqref{eq:mfista_monotone_selection} and~\eqref{eq:mfista_extrapolation}, after the domain and backtracking checks. Its candidate-selection branch is not part of the basic FISTA error recursion proved below.

\begin{algorithm}[htbp]
\caption{Gradient-enhanced Douglas--Rachford splitting}\label{alg:GE_DR}
\begin{algorithmic}[1]
\Require Endpoints, $\mathcal Y_h$, auxiliary $\overline{\mathbf z}_h^{(0)}$, fixed $\gamma>0$, $0<\alpha<2$, stopping parameters, $K_{\max}$.
\State Set $\mathbf z_h^{(0)}=\widetilde P_h\overline{\mathbf z}_h^{(0)}$.
\For{$k=0,\ldots,K_{\max}-1$}
\State Compute $\mathbf p_h^{(k)}=\operatorname{Prox}_{\gamma\mathcal Y_h}^{W_h}(2\mathbf z_h^{(k)}-\overline{\mathbf z}_h^{(k)})$.
\State Set $\overline{\mathbf z}_h^{(k+1)}=\overline{\mathbf z}_h^{(k)}+\alpha(\mathbf p_h^{(k)}-\mathbf z_h^{(k)})$.
\State Set $\mathbf z_h^{(k+1)}=\widetilde P_h\overline{\mathbf z}_h^{(k+1)}$.
\State Check the splitting gap and constraint diagnostics; return if satisfied.
\EndFor
\State Return the last projected and energy-proximal outputs with an iteration-limit status.
\end{algorithmic}
\end{algorithm}

For ISTA/FISTA, stopping requires both a small relative step and a small constraint residual, together with endpoint, mass, and domain checks:
\begin{equation}\label{eq:stopping_diagnostics}
\begin{split}
r_{\rm step}^{(k+1)}&=\frac{\|\mathbf z_h^{(k+1)}-\mathbf z_h^{(k)}\|_h}{\max\{1,\|\mathbf z_h^{(k)}\|_h\}},\qquad
r_{\mathcal C}(\mathbf z_h)=\frac{\|A_h\mathbf z_h-\mathbf b_h\|_2}{\max\{1,\|\mathbf b_h\|_2\}},\\
r_{\rm mass}(\mathbf z_h)&=\max_{0\le j\le N_t}\left|\sum_iA_i\rho_{h,ji}-1\right|.
\end{split}
\end{equation}
Each diagnostic is compared with a specified tolerance; the minimum density, endpoint discrepancy, compatibility defect~\eqref{eq:compatibility_defect}, and linear residual are recorded. For DR, the relative gap $\|\mathbf p_h^{(k)}-\mathbf z_h^{(k)}\|_h/\max\{1,\|\mathbf z_h^{(k)}\|_h\}$ replaces the step test, and continuity is checked for both outputs. Since only the energy-proximal output is automatically in the kinetic-energy domain, their difference and any negative projected density must be reported. These are termination diagnostics, not optimality certificates for the recovered iteration.

\section{Discretization error analysis}\label{sec:error_analysis}
The analysis separates the accuracy of a recovered Poisson solve from perturbations of the outer iteration. We first specify interpolation and lifting, then derive projection and residual identities, and finally estimate finite sequences of basic ISTA/FISTA and Douglas--Rachford steps. None of these statements identifies a recovered map with an exact discrete projection.

\subsection{Lifting, interpolation, and reference problems}
Let $p_h:\mathcal M_h\to\mathcal M$ be the geometric correspondence, which is a bijection for the mesh families considered here. For scalars, $T_hu=u\circ p_h$ and $T_h^{-1}u_h=u_h\circ p_h^{-1}$. For vectors, the lifted ambient field is projected onto $T\mathcal M$. We write $\mathcal J_h\mathbf z_h=T_h^{-1}\mathbf z_h$ for this pairwise lift and retain $T_h^{-1}$ for scalar and recovered-vector expressions. Define $\mathcal I_h\mathbf z$ by temporal and spatial nodal interpolation and by projection of momentum values onto the reconstructed nodal tangent planes. It does not reset density endpoints; the endpoint approximation is treated separately.

\begin{assumption}[Geometric and algebraic setting]\label{ass:geometry}
The triangulations are connected and shape regular, the recovery fits have the ranks specified in Section~\ref{sec:discretization}, and the geometric correspondence is uniformly regular. The lift is bounded on the ambient discrete pair space, and is norm-equivalent on $X_h$: for constants $C_{\rm lift},C_{\rm inv}$,
\[
\|\mathcal J_h\mathbf v_h\|_{\mathcal H}\le C_{\rm lift}\|\mathbf v_h\|_h,
\qquad
\|\mathbf v_h\|_h\le C_{\rm inv}\|\mathcal J_h\mathbf v_h\|_{\mathcal H}
\quad(\mathbf v_h\in X_h).
\]
Uniform statements require these constants to be independent of $h,\tau$. Such equivalence is compatible with the usual surface lifting estimates when the reconstructed nodal normals approach the smooth normal and $h$ is sufficiently small~\cite{dziuk2013surface,dong2020parametric}. It is not asserted for arbitrary degenerate patches.
\end{assumption}

\begin{lemma}[Interpolation]\label{lem:interpolation}
For scalar $u\in H^2(\mathcal T;H^2(\mathcal M))$, under the usual piecewise linear interpolation and lifting estimates,
\begin{equation}\label{eq:interpolation_bound}
\|u-T_h^{-1}I_\tau I_hT_hu\|_{L^2(\mathcal Q)}
\le C(h^2+\tau^2)\|u\|_{H^2(\mathcal T;H^2(\mathcal M))}.
\end{equation}
For a smooth pair write $\varepsilon_{I,h}(\mathbf z)=\|\mathbf z-\mathcal J_h\mathcal I_h\mathbf z\|_{\mathcal H}$. In addition to the componentwise interpolation bound, this quantity includes the error of the reconstructed tangent planes.
\end{lemma}
\begin{proof}
Decompose the scalar error into $u-I_\tau u$ and $I_\tau(u-T_h^{-1}I_hT_hu)$. The one-dimensional interpolation estimate bounds the first term by $C\tau^2\|\partial_{tt}u\|_{L^2(\mathcal Q)}$. Applying the spatial $L^2$ interpolation estimate at time nodes, followed by the temporal norm bound for piecewise linear interpolation, bounds the second term by $Ch^2\|u\|_{H^2(\mathcal T;H^2(\mathcal M))}$. The lifting constants are absorbed into $C$. For momentum, inserting the unprojected nodal interpolant separates its interpolation error from the nodal normal-projection error; the orthogonal projection onto the smooth tangent bundle cannot increase the latter comparison norm.
\end{proof}
This lemma concerns interpolation. An FDM--sFEM solution has an additional equation error and cannot be substituted for the interpolant in~\eqref{eq:interpolation_bound} without further analysis.

For a smooth input $\mathbf z$, let $\varphi[\mathbf z]$ solve~\eqref{eq:projection_weak_poisson}. Define $\varphi_h[\mathbf z]\in V_h$ to be the matched-data reference FDM--sFEM solution with load
\begin{equation}\label{eq:reference_load}
\ell_h^{\rm ref}(\mathbf z;v_h)
=\sum_jw_j\bigl(T_h[f(t_j)+b_j(a_0,a_1)],v_{h,j}\bigr)_{\mathcal M_h},
\end{equation}
where $f,a_0,a_1$ are the exact data in~\eqref{eq:Poisson-Equation} and $b_j$ has the signs in~\eqref{eq:discrete_boundary_load}. The load is restricted to $V_h$, equivalently centered. This reference uses exact input data, whereas~\eqref{eq:recovered_load} differentiates discrete data by recovery. Set
\begin{equation}\label{eq:recovery_error_definition}
\varepsilon_{{\rm rec},h}(\mathbf z)
=\|\mathbb G\varphi[\mathbf z]-\mathcal J_h\mathbb G_h\varphi_h[\mathbf z]\|_{\mathcal H}.
\end{equation}

\begin{proposition}[Recovery estimate used from the literature]\label{prop:imported_recovery}
Suppose the matched-data reference problem satisfies the geometric, mesh, and regularity hypotheses of the temporal and spatial recovery results of Jiang et al.~\cite{jiang2026fdm}. In particular, let $\mathcal M_h$ satisfy their $\mathcal O(h^{2\sigma})$ irregularity condition, with $\sigma>0$, and let $u=\varphi[\mathbf z]\in H^4(\mathcal T;L^2(\mathcal M))\cap H^2(\mathcal T;H^2(\mathcal M))$. Let the semidiscrete solution $u_N=\sum_j u^j\phi_j$ have the spatial regularity appearing below. In the mean-zero data regime of those estimates, their combined consequence is
\begin{equation}\label{equ:main_superconvergence_result}
\varepsilon_{{\rm rec},h}(\mathbf z)
\le C\bigl(h^2\mathcal A_1+h^{1+\min\{1,\sigma\}}\mathcal A_2+\tau^2\mathcal A_3\bigr),
\end{equation}
where, with all mixed norms over $\mathcal T\times\mathcal M$,
\begin{align*}
\mathcal A_1&=\|u_N\|_{L^2(\mathcal T;W^{3,\infty}(\mathcal M))}
+\|\partial_tu_N\|_{L^2(\mathcal T;H^2(\mathcal M))},\\
\mathcal A_2&=\|u_N\|_{L^2(\mathcal T;H^3(\mathcal M))}
+\|u_N\|_{L^2(\mathcal T;W^{2,\infty}(\mathcal M))},\\
\mathcal A_3&=\|u\|_{H^3(\mathcal T;L^2(\mathcal M))}
+\|\nabla_g\partial_{tt}u\|_{L^2(\mathcal Q)}
+\|\nabla_g\Delta_g^{-1}\partial_{tt}f\|_{L^2(\mathcal Q)}.
\end{align*}
Here $\Delta_g^{-1}$ acts on spatially mean-zero functions. Any Neumann-data lifting must satisfy the corresponding regularity hypotheses as well.
\end{proposition}
\begin{proof}
The temporal and spatial recovered-derivative estimates in~\cite{jiang2026fdm} bound their squared errors by sums with factors $h^4$, $h^{2+2\min\{1,\sigma\}}$, and $\tau^4$. Taking square roots and adding the component bounds gives~\eqref{equ:main_superconvergence_result}. This is a consequence of the cited recovery theory, not a new superconvergence theorem for the optimization method.
\end{proof}
The irregularity condition measures mildly structured meshes: outside a collection of triangles of total area $\mathcal O(h^{2\sigma})$, neighboring triangles satisfy the approximate-parallelogram condition used in the cited supercloseness estimate. Its chart and geometric requirements are part of that result. The bound is conditional on those requirements, not merely on shape regularity. When the displayed norms are bounded and $\sigma\ge1$, its recovery contribution is $\mathcal O(h^2+\tau^2)$.

For general projection inputs the source need not have zero spatial mean at every time. Then the inverse Laplacian in Proposition~\ref{prop:imported_recovery} may only be applied to the mean-zero component. The spatial mean of the potential instead satisfies a scalar temporal Neumann problem. If that mode, the boundary lifting, or the geometric transfer has not been covered by an applicable estimate, its error remains in $\varepsilon_{{\rm rec},h}$; formula~\eqref{equ:main_superconvergence_result} is not applied to it without justification. The perturbation results below use the definition~\eqref{eq:recovery_error_definition} and remain valid with these contributions retained.

\subsection{Projection error, data perturbation, and residuals}
Use $\|v_h\|_{a,h}=\mathfrak a_h(v_h,v_h)^{1/2}$ on $V_h$ and the corresponding dual norm $\|\cdot\|_{V_h^*}$. Introduce the explicit operator constants
\begin{equation}\label{eq:operator_constants}
\begin{split}
\kappa_h&=\sup_{0\ne v_h\in V_h}
\frac{\|\mathcal J_h\mathbb G_hv_h\|_{\mathcal H}}{\|v_h\|_{a,h}},\\
\Lambda_h&=\sup_{0\ne\mathbf v_h\in X_h}
\frac{\|\ell_h(\mathbf v_h;\cdot)-\ell_h(\mathbf0;\cdot)\|_{V_h^*}}{\|\mathbf v_h\|_h}.
\end{split}
\end{equation}
They are finite matrix-operator norms on any fixed admissible mesh. Their uniform boundedness is a separate stability requirement; polynomial preservation alone does not establish it. Define the matched-input load defect
\begin{equation}\label{eq:load_consistency}
d_h(\mathbf z)=\|\ell_h(\mathcal I_h\mathbf z;\cdot)-\ell_h^{\rm ref}(\mathbf z;\cdot)\|_{V_h^*}.
\end{equation}
It includes recovered differentiation of the input, endpoint approximation and normalization, and any difference in load transfer. This defect is distinct from derivative recovery of the potential.

Let $\widehat\varphi_h$ be an inexact solve of~\eqref{eq:Poisson-gradient_enhanced} for input $\mathbf v_h$, with residual
\[
s_h(v_h)=\mathfrak a_h(\widehat\varphi_h,v_h)-\ell_h(\mathbf v_h;v_h),
\qquad \varepsilon_{s,h}=\|s_h\|_{V_h^*}.
\]
In this display the bold $\mathbf v_h$ is the input pair and the unbold $v_h\in V_h$ is a scalar test function. To account for endpoint and tangency enforcement, set
\begin{equation}\label{eq:enforcement_error}
\varepsilon_{R,h}=\left\|\mathcal J_h\left[
\mathcal R_h(\mathbf v_h+\mathbb G_h\widehat\varphi_h)
-(\mathbf v_h+\mathbb G_h\widehat\varphi_h)\right]\right\|_{\mathcal H}.
\end{equation}
This is an explicitly defined correction norm, not an assumed higher-order term.

\begin{lemma}[Single recovered projection]\label{lem:projection_error}
Under Assumption~\ref{ass:geometry}, for smooth $\mathbf z$ and $\mathbf v_h\in X_h$, let $\widehat P_h\mathbf v_h=\mathcal R_h(\mathbf v_h+\mathbb G_h\widehat\varphi_h)$. Then
\begin{equation}\label{eq:projection_error_bound}
\begin{split}
\|P_{\mathcal C}\mathbf z-\mathcal J_h\widehat P_h\mathbf v_h\|_{\mathcal H}
\le{}& a_h\|\mathbf z-\mathcal J_h\mathbf v_h\|_{\mathcal H}
+\varepsilon_{{\rm rec},h}(\mathbf z)
+\kappa_h\bigl[d_h(\mathbf z)+\varepsilon_{s,h}\bigr]\\
&+(a_h-1)\varepsilon_{I,h}(\mathbf z)+\varepsilon_{R,h},\qquad
a_h=1+\kappa_h\Lambda_h C_{\rm inv}.
\end{split}
\end{equation}
\end{lemma}
\begin{proof}
Subtract the equations for $\widehat\varphi_h$ and $\varphi_h[\mathbf z]$. Testing the difference with itself gives
\[
\|\widehat\varphi_h-\varphi_h[\mathbf z]\|_{a,h}
\le \Lambda_h\|\mathbf v_h-\mathcal I_h\mathbf z\|_h+d_h(\mathbf z)+\varepsilon_{s,h}.
\]
The affine endpoint part cancels when taking the difference of two recovered loads. Moreover,
\[
\|\mathbf v_h-\mathcal I_h\mathbf z\|_h
\le C_{\rm inv}\bigl(\|\mathcal J_h\mathbf v_h-\mathbf z\|_{\mathcal H}+\varepsilon_{I,h}(\mathbf z)\bigr).
\]
Use $P_{\mathcal C}\mathbf z=\mathbf z+\mathbb G\varphi[\mathbf z]$, insert the reference recovered derivative, and apply~\eqref{eq:operator_constants}. The final enforcement changes the lifted output by exactly the term in~\eqref{eq:enforcement_error}. The triangle inequality yields~\eqref{eq:projection_error_bound}.
\end{proof}

There is also an algebraic residual identity which does not require smooth input. Choose coordinates on $V_h$ so that its positive definite Poisson matrix is $\mathscr L_h$. Write the load vector as $D_h\mathbf z_h+\mathbf c_h$ and the recovered-gradient matrix as $\mathscr G_h$. Here $\mathbf c_h$ is the fixed affine part of the load, not the scalar compatibility defect $c_h(\mathbf z_h)$. For an exact scalar solve define
\[
U_h=I+\mathscr G_h\mathscr L_h^{-1}D_h,\qquad
\mathbf q_h=\mathscr G_h\mathscr L_h^{-1}\mathbf c_h,
\quad \mathcal R_h\mathbf w=R_h\mathbf w+\mathbf r_h.
\]
Then $\widetilde P_h\mathbf z_h=H_h\mathbf z_h+\mathbf k_h$, with $H_h=R_hU_h$ and $\mathbf k_h=R_h\mathbf q_h+\mathbf r_h$.
\begin{proposition}[Residual and idempotence defects]\label{prop:residual_identity}
For the specified weak constraint matrix $A_h$,
\begin{align}
A_h\widetilde P_h\mathbf z_h-\mathbf b_h
&=A_hH_h\mathbf z_h+A_h\mathbf k_h-\mathbf b_h,\label{eq:constraint_defect_identity}\\
\widetilde P_h(\widetilde P_h\mathbf z_h)-\widetilde P_h\mathbf z_h
&=(H_h^2-H_h)\mathbf z_h+H_h\mathbf k_h.\label{eq:idempotence_defect}
\end{align}
Thus feasibility of every output requires $A_hH_h=0$ and $A_h\mathbf k_h=\mathbf b_h$. Idempotence requires $H_h^2=H_h$ and $H_h\mathbf k_h=0$. These conditions alone do not prove orthogonality in $W_h$.
\end{proposition}
\begin{proof}
Substitute the affine representation into each left-hand side and collect the linear and constant terms. The necessary and sufficient identities follow by varying $\mathbf z_h$. Orthogonality would additionally require the appropriate $W_h$-self-adjoint linear part and the correct range, as in~\eqref{eq:exact_discrete_projection}.
\end{proof}
This gives direct matrix checks for conservation and repeated application of the recovered map. It also explains why centering a Poisson load cannot by itself establish the projection property.

\subsection{Finite-iteration bounds for basic ISTA and FISTA}
Consider the exact continuous reference iteration and its discrete counterpart with the same prescribed parameter sequences:
\begin{equation}\label{eq:paired_iterations}
\begin{aligned}
\mathbf z^{(k+1/2)}&=\overline{\mathbf z}^{(k)}-\eta^{(k)}\nabla\mathcal Y(\overline{\mathbf z}^{(k)}),
&\mathbf z^{(k+1)}&=P_{\mathcal C}\mathbf z^{(k+1/2)},\\
\mathbf z_h^{(k+1/2)}&=\overline{\mathbf z}_h^{(k)}-\eta^{(k)}\nabla_h\mathcal Y_h(\overline{\mathbf z}_h^{(k)}),
&\mathbf z_h^{(k+1)}&=\widehat P_h\mathbf z_h^{(k+1/2)}.
\end{aligned}
\end{equation}
For ISTA the barred variables equal the current variables. For basic FISTA they use the common coefficient $\omega^{(k)}=(\tau^{(k)}-1)/\tau^{(k+1)}$. Denote
\[
\mathbf E^{(k)}=\mathbf z^{(k)}-\mathcal J_h\mathbf z_h^{(k)},\quad
e^{(k)}=\|\mathbf E^{(k)}\|_{\mathcal H},\quad
\overline e^{(k)}=\|\overline{\mathbf z}^{(k)}-\mathcal J_h\overline{\mathbf z}_h^{(k)}\|_{\mathcal H}.
\]
The same notation with $k+1/2$ is used for the intermediate pair.

\begin{assumption}[Well-defined explicit steps]\label{ass:explicit_steps}
For the finite index range considered, the reference fields and their interpolants have the regularity needed in Lemma~\ref{lem:projection_error}. The discrete extrapolated states and interpolated reference extrapolated states lie in a convex region where $\nabla_h\mathcal Y_h$ is $L_h$-Lipschitz in $\|\cdot\|_h$. All evaluated energy gradients exist. For the kinetic energy with local interactions, positive lower density bounds, upper momentum bounds, and bounded interaction second derivatives provide such a region. These are conditions on the trajectories, not properties proved here for all initial data.
\end{assumption}
Define the energy-gradient consistency error
\[
\varepsilon_{g,h}(\mathbf z)
=\|\nabla\mathcal Y(\mathbf z)-\mathcal J_h\nabla_h\mathcal Y_h(\mathcal I_h\mathbf z)\|_{\mathcal H},
\qquad L_h^\sharp=C_{\rm lift}L_hC_{\rm inv}.
\]

\begin{lemma}[Explicit half-step]\label{lem:half_step}
Under Assumptions~\ref{ass:geometry} and~\ref{ass:explicit_steps},
\begin{equation}\label{eq:gradient_descent_error_proof1}
e^{(k+1/2)}\le(1+\eta^{(k)}L_h^\sharp)\overline e^{(k)}
+\eta^{(k)}\bigl[\varepsilon_{g,h}(\overline{\mathbf z}^{(k)})
+L_h^\sharp\varepsilon_{I,h}(\overline{\mathbf z}^{(k)})\bigr].
\end{equation}
\end{lemma}
\begin{proof}
Subtract the two half-step equations in~\eqref{eq:paired_iterations}. The initial term is the error at the $k$th extrapolated point. Insert $\nabla_h\mathcal Y_h(\mathcal I_h\overline{\mathbf z}^{(k)})$ into the gradient difference. Its consistency term is $\varepsilon_{g,h}$, while the remaining term is bounded by
\[
C_{\rm lift}L_h\|\overline{\mathbf z}_h^{(k)}-\mathcal I_h\overline{\mathbf z}^{(k)}\|_h
\le L_h^\sharp\bigl(\overline e^{(k)}+\varepsilon_{I,h}(\overline{\mathbf z}^{(k)})\bigr).
\]
The triangle inequality proves the claim. No discrete consistency estimate is inferred solely from continuous Lipschitz continuity.
\end{proof}

\begin{theorem}[Conditional finite-iteration estimate]\label{thm:finite_iterations}
Fix $K\ge1$ and suppose Assumptions~\ref{ass:geometry} and~\ref{ass:explicit_steps} hold through the corresponding half-steps. Let
\begin{equation}\label{eq:iteration_remainder}
\begin{split}
\delta_h^{(k)}={}&a_h\eta^{(k)}
\bigl[\varepsilon_{g,h}(\overline{\mathbf z}^{(k)})+L_h^\sharp\varepsilon_{I,h}(\overline{\mathbf z}^{(k)})\bigr]
+\varepsilon_{{\rm rec},h}(\mathbf z^{(k+1/2)})\\
&+\kappa_h\bigl[d_h(\mathbf z^{(k+1/2)})+\varepsilon_{s,h}^{(k)}\bigr]
+(a_h-1)\varepsilon_{I,h}(\mathbf z^{(k+1/2)})+\varepsilon_{R,h}^{(k)}.
\end{split}
\end{equation}
Then
\begin{equation}\label{eq:one_step_recurrence}
e^{(k+1)}\le a_h(1+\eta^{(k)}L_h^\sharp)\overline e^{(k)}+\delta_h^{(k)}.
\end{equation}
For basic FISTA let $\Omega=\max_{0\le k<K}|\omega^{(k)}|$, and for ISTA set $\Omega=0$. If both reference and discrete initial extrapolated points equal their initial iterates, then
\begin{equation}\label{eq:finite_iteration_bound}
\max_{0\le k\le K}e^{(k)}
\le M_h^K e^{(0)}+\sum_{k=0}^{K-1}M_h^{K-1-k}\delta_h^{(k)},\qquad
M_h=\max\left\{1,(1+2\Omega)\max_{k<K}a_h(1+\eta^{(k)}L_h^\sharp)\right\}.
\end{equation}
\end{theorem}
\begin{proof}
Apply Lemma~\ref{lem:projection_error} to $\mathbf z^{(k+1/2)}$ and $\mathbf z_h^{(k+1/2)}$, then use Lemma~\ref{lem:half_step}. This proves~\eqref{eq:one_step_recurrence}. For FISTA, linearity of the lift gives the exact identity
\[
\overline{\mathbf E}^{(k+1)}=(1+\omega^{(k)})\mathbf E^{(k+1)}-\omega^{(k)}\mathbf E^{(k)}.
\]
Thus $\overline e^{(k+1)}\le(1+|\omega^{(k)}|)e^{(k+1)}+|\omega^{(k)}|e^{(k)}$. Let $F_k=\max_{0\le j\le k}e^{(j)}$. The initialization and this inequality imply $F_{k+1}\le M_hF_k+\delta_h^{(k)}$; this also holds for ISTA. Iterating the scalar inequality gives~\eqref{eq:finite_iteration_bound}.
\end{proof}

When Proposition~\ref{prop:imported_recovery} applies with uniformly bounded regularity factors, its contribution to~\eqref{eq:iteration_remainder} is $\mathcal O(h^{1+\min\{1,\sigma\}}+\tau^2)$. The remaining terms stay explicit. A mesh-independent finite-$K$ rate additionally requires bounds for $\kappa_h,\Lambda_h,L_h^\sharp$, input-load and energy consistency, endpoint enforcement, and solver errors. The theorem does not establish those bounds by assuming that the final iterate is already second order. Nor is $M_h^K$ uniform as $K\to\infty$.

The theorem concerns prescribed common step and extrapolation sequences. It can be used for a fixed realized sequence of accepted parameters, with a reference iteration driven by that same sequence. Independently adapting the continuous and discrete algorithms may lead to different restarts or MFISTA selections and requires additional branch-sensitive analysis. Finally, discretization and optimization errors are distinct: if $\mathbf z^\star$ is a minimizer,
\[
\|\mathcal J_h\mathbf z_h^{(K)}-\mathbf z^\star\|_{\mathcal H}
\le e^{(K)}+\|\mathbf z^{(K)}-\mathbf z^\star\|_{\mathcal H}.
\]
The second term is not bounded by the derivative-recovery theorem.

\subsection{Douglas--Rachford perturbations}
Let $P=P_{\mathcal C}$ and $Q_\gamma=\operatorname{Prox}_{\gamma\mathcal Y}$ on $\mathcal H$. Assume $\mathcal Y$ is proper, lower semicontinuous, and convex, so its proximal map is firmly nonexpansive. For fixed $\gamma>0$ and $0<\alpha<2$, the exact auxiliary-variable map is
\[
\mathcal D\mathbf z=\mathbf z+\alpha[Q_\gamma(2P\mathbf z-\mathbf z)-P\mathbf z]
=\left(1-\frac\alpha2\right)\mathbf z+\frac\alpha2(2Q_\gamma-I)(2P-I)\mathbf z.
\]
The reflected proximal maps are nonexpansive, and therefore so is $\mathcal D$. This statement uses exact $P$ and $Q_\gamma$.

For Algorithm~\ref{alg:GE_DR}, include scalar-solver errors in $\widehat P_h$ and any local proximal-solver errors in $\widehat Q_{\gamma,h}$. At step $k$ define
\begin{align*}
\pi_h^{(k)}&=\|\mathcal J_h\widehat P_h\overline{\mathbf z}_h^{(k)}-P(\mathcal J_h\overline{\mathbf z}_h^{(k)})\|_{\mathcal H},\\
\mathbf v_h^{(k)}&=2\mathbf z_h^{(k)}-\overline{\mathbf z}_h^{(k)},\qquad
q_h^{(k)}=\|\mathcal J_h\widehat Q_{\gamma,h}\mathbf v_h^{(k)}-Q_\gamma(\mathcal J_h\mathbf v_h^{(k)})\|_{\mathcal H}.
\end{align*}
These are total lifted map defects. Lemma~\ref{lem:projection_error} can resolve the projection defect when its input has the required smoothness; it is not used to impose that smoothness on an arbitrary discrete auxiliary variable.

\begin{theorem}[Finite-step DR perturbation bound]\label{thm:DR_perturbation}
For exact and discrete auxiliary sequences with the same $\gamma,\alpha$, set $\overline e^{(k)}=\|\overline{\mathbf z}^{(k)}-\mathcal J_h\overline{\mathbf z}_h^{(k)}\|_{\mathcal H}$. Then
\begin{equation}\label{eq:DR_error_bound}
\overline e^{(K)}\le\overline e^{(0)}+\alpha\sum_{k=0}^{K-1}(3\pi_h^{(k)}+q_h^{(k)}),\qquad
\|\mathbf z^{(K)}-\mathcal J_h\mathbf z_h^{(K)}\|_{\mathcal H}\le\overline e^{(K)}+\pi_h^{(K)}.
\end{equation}
\end{theorem}
\begin{proof}
Fix a lifted auxiliary input $\mathbf w=\mathcal J_h\overline{\mathbf z}_h^{(k)}$. The difference between the lifted discrete projected point and $P\mathbf w$ is at most $\pi_h^{(k)}$, so the reflected arguments differ by at most $2\pi_h^{(k)}$. Nonexpansiveness of $Q_\gamma$ and the definition of $q_h^{(k)}$ bound the energy-proximal output discrepancy by $2\pi_h^{(k)}+q_h^{(k)}$. The subtraction of the projected point contributes one more $\pi_h^{(k)}$. Consequently the lifted discrete auxiliary update differs from $\mathcal D\mathbf w$ by at most $\alpha(3\pi_h^{(k)}+q_h^{(k)})$. Nonexpansiveness of $\mathcal D$ gives the one-step error bound, which sums to the first assertion. The second follows from nonexpansiveness of $P$ and the definition of $\pi_h^{(K)}$.
\end{proof}
This bound makes accumulated projection and local-proximal defects explicit. It does not prove that these defects decrease or are summable on a fixed mesh. Results for inexact proximal-gradient methods impose specific error criteria and decay conditions~\cite{schmidt2011inexact}; their objective-based approximate-proximal conditions cannot simply be imposed on an infeasible point for an indicator functional, whose value there is infinite. Our finite-iteration estimates consequently do not assert global convergence or an unconditional accelerated rate for the recovered algorithms.

\FloatBarrier
\Needspace{10\baselineskip}
\section{Numerical Experiments}\label{sec:numerics}
This section retains the available density-evolution figures and refinement values. They provide illustrations of the intended computations; the original solver configurations and raw error data are not available with the manuscript. In particular, these records do not independently validate the corrected boundary treatment, safeguards, or perturbation bounds above. We restrict the discussion to the displayed behavior and recorded values, and distinguish it from a reproducible convergence or timing study.

\subsection{Gradient enhanced ISTA}
We investigate the performance of the gradient-enhanced ISTA method in Algorithm~\ref{alg:GradientEnhanced_ISTA} for solving dynamic optimal transport on the sphere. The initial distribution at $t=0$ and the terminal distribution at $t=1$ are chosen as spherical Gaussian distributions with centers
$$\mu_0=(-1,-1,0.5), \qquad \mu_1=(0.6,-1,0.5),$$
with the reported Gaussian parameter $\sigma=0.1$. As shown in Figure~\ref{fig:gradient_enhanced_ISTA}, the recorded results use three different stepsizes $\eta$. Their density patterns differ, especially for $\eta=1.8\times10^{-4}$. These pictures motivate careful stepsize selection and residual monitoring, but do not alone identify the source of the discrepancy or quantify an accuracy loss.
\begin{figure}[H]
    \centering
    \begin{subfigure}{\linewidth}
        \centering
        \includegraphics[width=\linewidth]{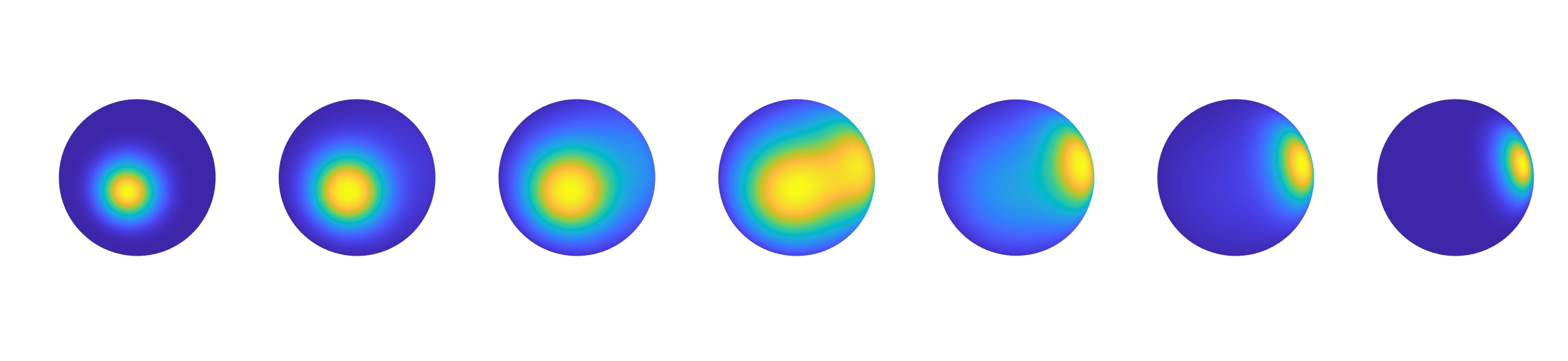}
        \caption{Stepsize $\eta=1\times10^{-4}$.}
        \label{fig:PGDstep1e-4}
    \end{subfigure}
    \begin{subfigure}{\linewidth}
        \centering
        \includegraphics[width=\linewidth]{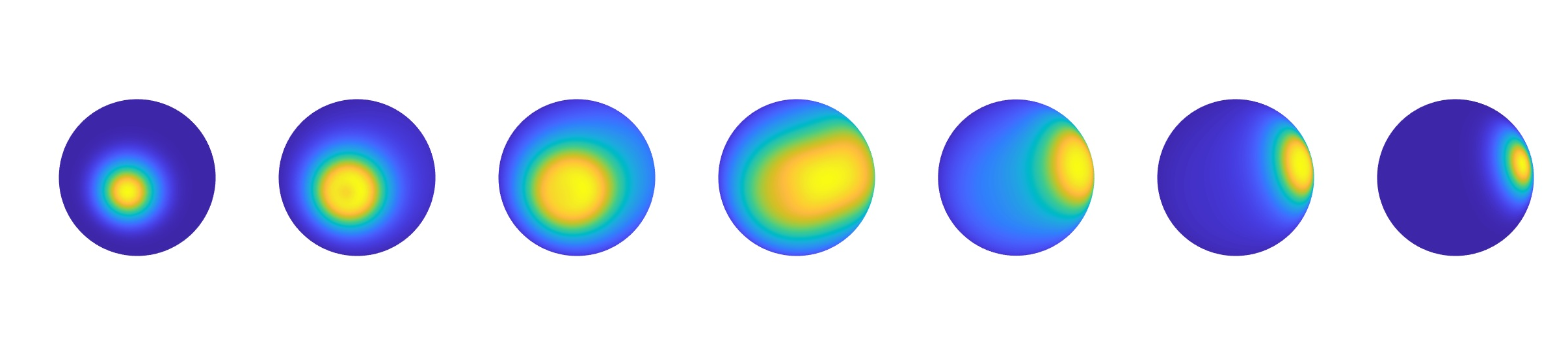}
        \caption{Stepsize $\eta=1.8\times10^{-4}$.}
        \label{fig:PGDstep1.8e-4}
    \end{subfigure}
    \begin{subfigure}{\linewidth}
        \centering
        \includegraphics[width=\linewidth]{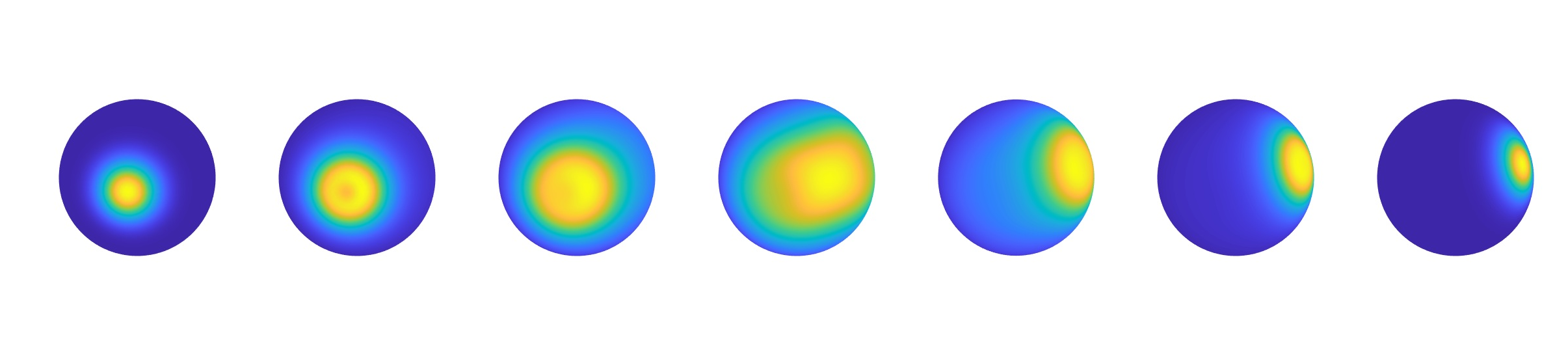}
        \caption{Stepsize $\eta=2\times10^{-4}$.}
        \label{fig:PGDstep2e-4}
    \end{subfigure}
    \caption{Solving dynamic OT on sphere using gradient enhanced ISTA with different stepsizes.}
    \label{fig:gradient_enhanced_ISTA}
\end{figure}

\subsection{FISTA}
We also investigate the performance of the gradient-enhanced FISTA method in Algorithm~\ref{alg:Gradient_enhanced_FISTA} for solving dynamic optimal transport on the sphere. The initial distribution at $t=0$ and the terminal distribution at $t=1$ are chosen as spherical Gaussian distributions with centers
$$\mu_0=(0,0,1), \qquad \mu_1=(0,0,-1),$$
with the reported Gaussian parameter $\sigma=0.1$. As shown in Figure~\ref{fig:gradient_enhanced_FISTA}, we compare the numerical results obtained with two different stepsizes $\eta$. The two displayed density sequences appear qualitatively similar. The record associates these calculations with backtracking, but does not provide the accepted stepsizes or backtracking history. Since the endpoints also differ from the preceding ISTA example, these figures are not a controlled comparison of acceleration or stepsize independence.
\begin{figure}[H]
    \centering
    \begin{subfigure}{\linewidth}
        \centering
        \includegraphics[width=\linewidth]{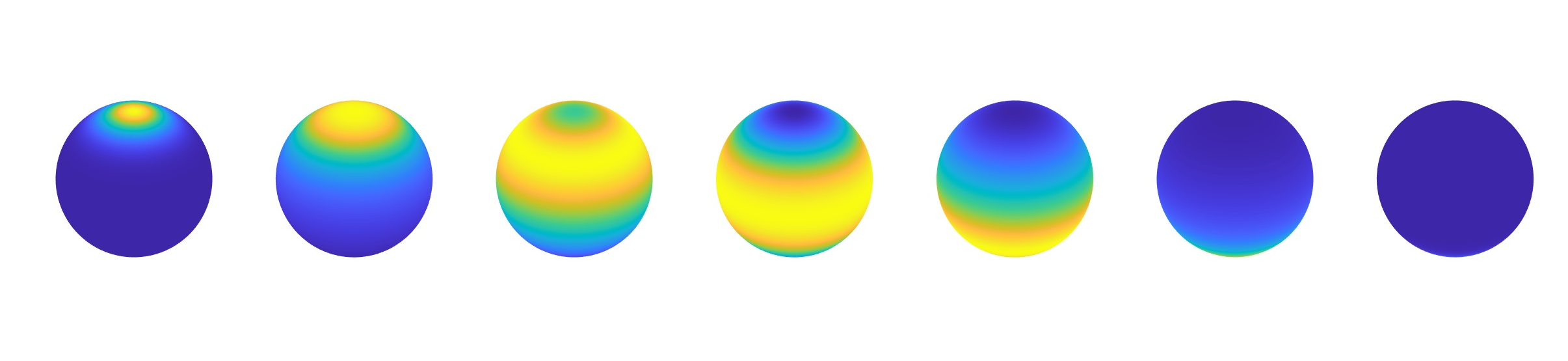}
        \caption{Stepsize $\eta=5e-4$.}
        \label{fig:FISTAstep1e-4}
    \end{subfigure}
    \begin{subfigure}{\linewidth}
        \centering
        \includegraphics[width=\linewidth]{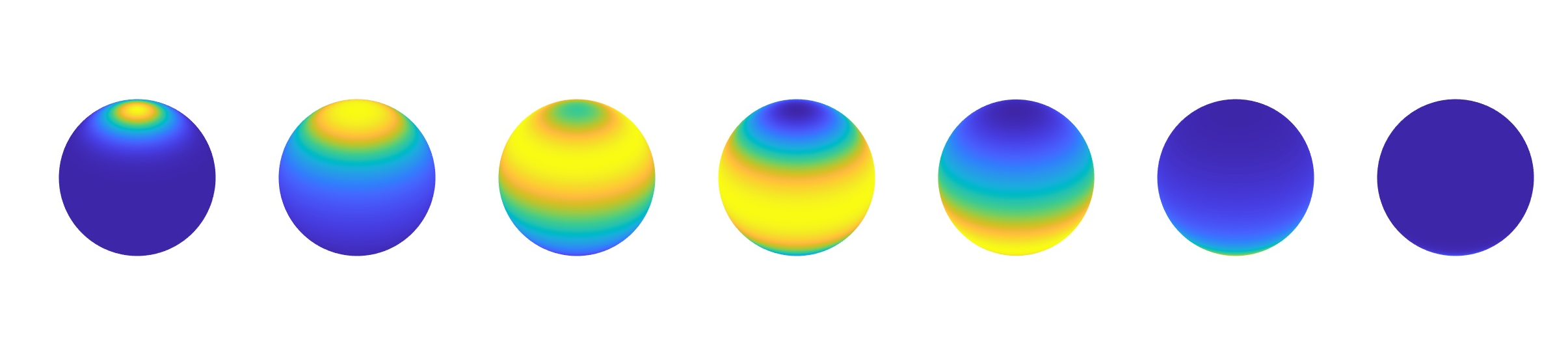}
        \caption{stepsize $\eta=2e-3$.}
        \label{fig:fista_step_2e-3}
    \end{subfigure}
    \caption{Solving dynamic OT on sphere using gradient enhanced FISTA with different stepsizes.}
    \label{fig:gradient_enhanced_FISTA}
\end{figure}

\subsection{Mean Field Planning}
Figure~\ref{fig:gradient_enhanced_FISTA_MFP} compares the recorded spherical density sequences for four local interaction integrands $F_E$: zero, $\rho^2/2$, $\rho\log\rho$, and $1/\rho$. With a fixed mass, convex density penalties can influence concentration and the occupation of low-density regions. The entropy integrand has the continuous extension $0\log0=0$, although its derivative is singular at zero; the reciprocal integrand has value $+\infty$ at zero. Consequently, their explicit-gradient evaluations require different domain checks from the zero-interaction OT proximal step. The pictures illustrate differences between the recorded sequences, without determining the relative contributions of the cost, its weight, and the numerical parameters.

\begin{figure}[H]
    \centering
    \begin{subfigure}{\linewidth}
        \centering
        \includegraphics[width=\linewidth]{figures/fistaTopDown2e-3.pdf}
        \caption{$F_{E}=0$}
        \label{fig:mfp_zero}
    \end{subfigure}
    \begin{subfigure}{\linewidth}
        \centering
        \includegraphics[width=\linewidth]{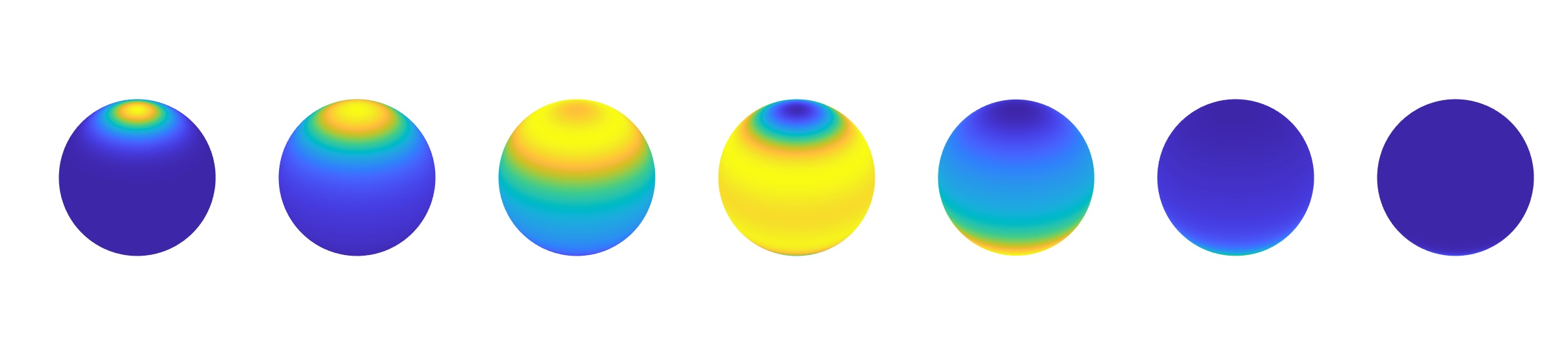}
        \caption{$F_{E}=\frac{\rho^2}{2}$}
        \label{fig:mfp_quadratic}
    \end{subfigure}
    \begin{subfigure}{\linewidth}
        \centering
        \includegraphics[width=\linewidth]{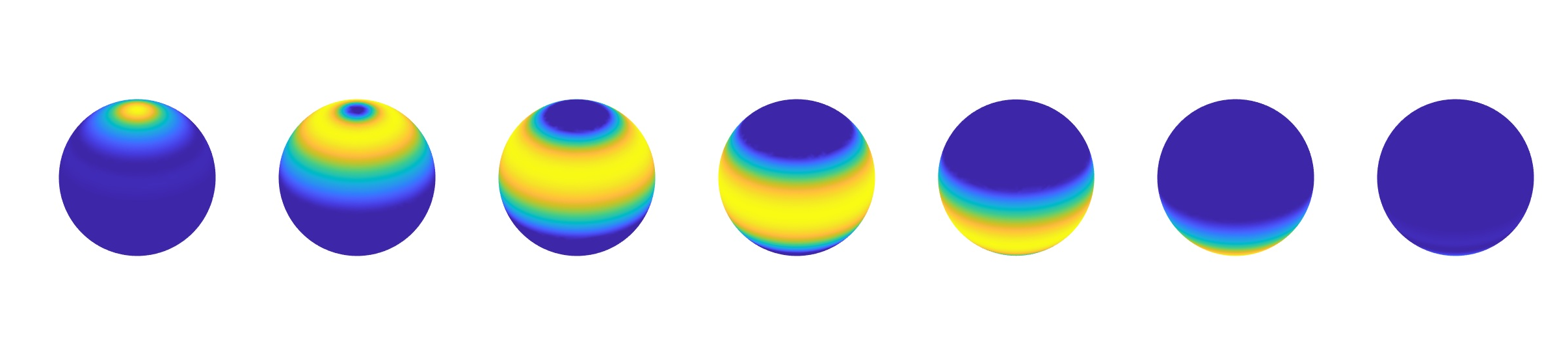}
        \caption{$F_{E}=\rho\log \rho,\ \rho>0$}
        \label{fig:mfp_entropy}
    \end{subfigure}
    \begin{subfigure}{\linewidth}
        \centering
        \includegraphics[width=\linewidth]{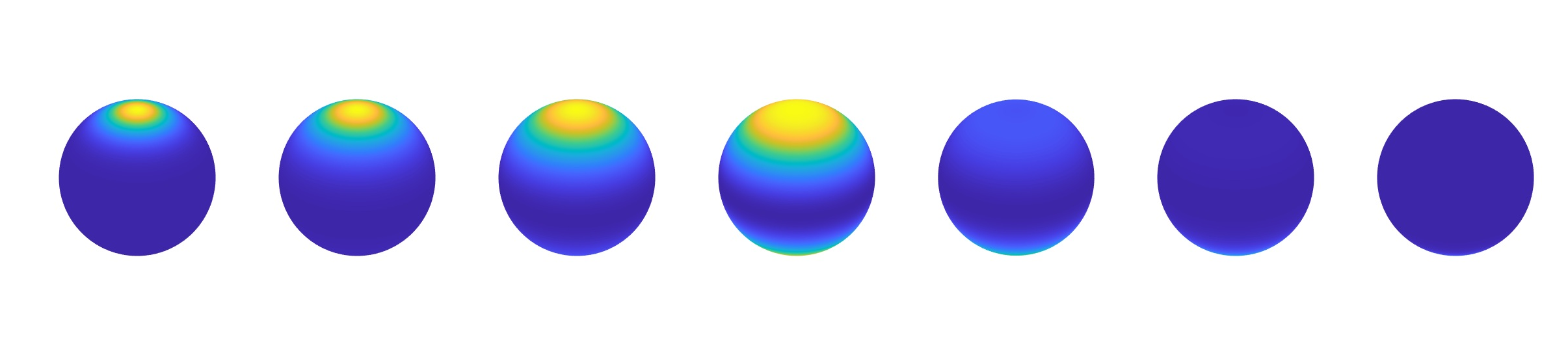}
        \caption{$F_{E}=\frac{1}{\rho},\ \rho>0$}
        \label{fig:mfp_reciprocal}
    \end{subfigure}
    \caption{Solving mean field planning on sphere using gradient enhanced FISTA with different extra cost functions $F_{E}$.}
    \label{fig:gradient_enhanced_FISTA_MFP}
\end{figure}
\subsection{Proximal Splitting}
We next test the proximal splitting scheme on a surface with more complicated geometry. The example uses the Enzensberger--Stern algebraic surface, defined as the zero level set of
\[
\phi(\mathbf{x}) = 400\left(x_1^2x_2^2+x_2^2x_3^2+x_1^2x_3^2\right)
-\left(1-x_1^2-x_2^2-x_3^2\right)^3-40.
\]
High-curvature regions make geometric approximation and local recovery conditioning relevant. The recorded initial and terminal distributions are indicator-type densities supported near two high-curvature corners. Figure~\ref{fig:DR_splitting_stern} shows the redistribution of density across the surface. The exact support regions, normalization and solver tolerances are needed to reproduce this calculation. Density snapshots alone do not certify transport optimality or identify individual geodesic trajectories; the nonsmooth endpoints also lie outside the smooth-data recovery assumptions.
\begin{figure}[htbp]
    \centering
    \includegraphics[width=1\linewidth]{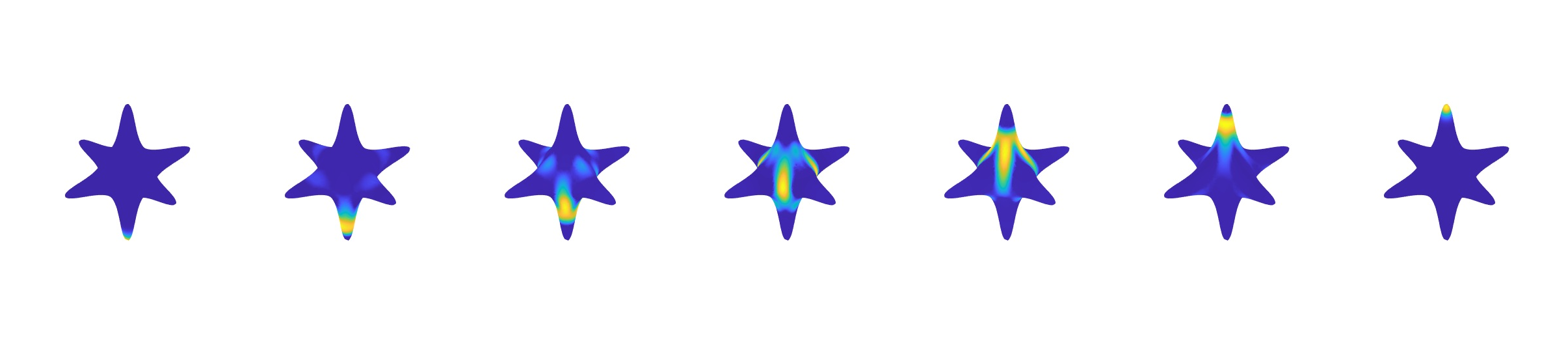}

    \caption{Solving dynamic OT on Enzensberger-Stern star algebraic surface using Douglas-Rachford splitting}
    \label{fig:DR_splitting_stern}
\end{figure}
\subsection{Recorded refinement values}
Table~\ref{tab:recorded_errors} preserves the reported errors and orders for FISTA, ADMM, Douglas--Rachford splitting, and their gradient-enhanced counterparts. The method references identify the associated algorithmic literature~\cite{yu2024fast,lavenant2018dynamical,dong2024admm,papadakis2014proximal}; the numbers are part of the manuscript's experiment record, rather than values reproduced from those publications. The recorded resolution index is denoted by $N$. Its relation to $h$ and $N_t$, the error norm, and the reference solution remain unspecified in that record, so the table is not used as a verification of Theorem~\ref{thm:finite_iterations}.

If consecutive resolution levels halve the relevant mesh parameter, the customary observed order is $\log_2(e_N/e_{2N})$. Recomputing this quantity from the displayed errors gives values close to the recorded orders; small differences in the last digits require the unrounded data. The GE errors decrease faster over several of the displayed refinements, although all three GE variants have larger errors at $N=8$ than their associated baselines. The last ADMM pair has an observed order near two. These records support neither a uniform advantage at every resolution nor a common first-order characterization of all baseline methods.

\begin{table}[htbp]
\centering
\caption{Recorded error values and reported orders. The resolution indices and numerical entries are retained from the existing record; a specified norm, reference solution, and grid relation are required for a reproducible convergence claim.}
\label{tab:recorded_errors}
\small
\setlength{\tabcolsep}{4pt}
\resizebox{\linewidth}{!}{
 \renewcommand{\arraystretch}{1.25}
\begin{tabular}{|l|c|c|c|c|c|c|c|c|}
\hline
\multirow{2}{*}{Method} 
& \multicolumn{2}{c|}{$N=8$} 
& \multicolumn{2}{c|}{$N=16$} 
& \multicolumn{2}{c|}{$N=32$} 
& \multicolumn{2}{c|}{$N=64$} \\
\cline{2-9}
& error & order 
& error & order 
& error & order 
& error & order \\
\hline
FISTA\cite{yu2024fast} &2.185e-1 &-- &9.857e-2 & 1.149  &4.286e-2 & 1.202 &1.931e-2 & 1.151\\
\hline
GE-FISTA & 2.731e-1 & -- & 5.848e-2 & 2.223 & 1.558e-2 & 1.908 & 3.799e-3 & 2.036 \\
\hline
ADMM\cite{lavenant2018dynamical} &3.012e-1 &-- &9.146e-2 &1.720 &3.281e-2 &1.479 &8.075e-3 & 2.023\\
\hline
GE-ADMM\cite{dong2024admm} & 3.406e-1 & -- & 6.225e-2 & 2.452 & 1.395e-2 & 2.157 & 2.903e-3 & 2.265 \\
\hline
DR-splitting\cite{papadakis2014proximal} &2.153e-1 &-- &8.965e-2 & 1.264 &3.871e-2 & 1.210 &1.889e-2 & 1.035\\
\hline
GE-DR-splitting & 2.691e-1&-- &6.228e-2 &2.111 &1.366e-2 &2.189 &2.793e-3 &2.290 \\
\hline
\end{tabular}}
\end{table}

\FloatBarrier
\Needspace{12\baselineskip}
\section{Conclusions}\label{sec:conclusion}
We have formulated a common gradient-enhanced approximation of the continuity-constraint projection for variational mean field planning on surfaces. The construction retains the time finite differences and surface finite elements of the Poisson solver, and uses temporal PPR and spatial PPPR to update density and momentum. Its integration into ISTA, FISTA, and Douglas--Rachford makes the same recovered differential quantities available to explicit-gradient and implicit-proximal iterations.

The analysis distinguishes the recovered update from an exact projection. It gives a matched-data projection error decomposition, algebraic tests for constraint and idempotence defects, a conditional finite-iteration estimate for basic ISTA/FISTA, and a finite-step perturbation bound for Douglas--Rachford. Higher-order recovered derivatives contribute to these estimates under the stated regularity and mesh hypotheses. Input differentiation, endpoint enforcement, geometric approximation, and solver errors must also be controlled before inferring an overall discretization rate or a convergent outer iteration.

The retained numerical illustrations show the intended surface transport and interaction-cost settings, while the recorded refinement values motivate further quantitative verification. A complete assessment requires the corresponding error definitions, reference solutions, conservation residuals, and computational costs. Uniform stability of the recovered projection, treatment of nonsmooth or vacuum data, and analysis of adaptive branch changes remain directions for further work.

\Needspace{10\baselineskip}
\bibliographystyle{plain}
\bibliography{references}

\end{document}